\documentclass[preprint,12pt]{elsarticle}

\usepackage{booktabs}   
\usepackage{multirow}   
\usepackage{amsmath,amsthm,amssymb}
\usepackage{graphicx}
\usepackage{booktabs}
\usepackage{hyperref}
\usepackage{cleveref}
\usepackage{algorithm}
\usepackage{algpseudocode}
\usepackage{subcaption}
\usepackage{enumitem}
\usepackage{lineno}
\newtheorem{assumption}{Assumption}
\newtheorem{theorem}{Theorem}
\newtheorem{lemma}{Lemma}
\newtheorem{corollary}{Corollary}
\newtheorem{proposition}{Proposition}

\theoremstyle{definition}
\newtheorem{definition}{Definition}

\theoremstyle{remark}
\newtheorem{remark}{Remark}

\journal{Elsevier}

\begin{document}

\begin{frontmatter}

\title{GK-Mapper: A Stability Framework for Gustafson-Kessel Fuzzy Mapper Graphs}

\author[ismd]{Annesha Sen}
\ead{23dr0026@iitism.ac.in}

\author[ismd]{Shivam Singh}
\ead{24dr0172@iitism.ac.in }

\author[ismd]{S.~P.~Tiwari\corref{cor1}}
\ead{sptiwari@iitism.ac.in}

\cortext[cor1]{Corresponding author.}

\address[ismd]{Department of Mathematics \& Computing,
               Indian Institute of Technology (ISM),
               Dhanbad-826004, India}

\begin{abstract}
Topological Data Analysis is the  field that uses algebraic topology for data analysis, with the Mapper Algorithm that studies the structure of data after reducing the dimension of the dataset. There are several variants of Mapper, like Conventional Mapper, F-Mapper, and Shape Fuzzy-C Means Mapper. In this article, we extend the idea of the Shape Fuzzy-C Means Mapper graphs by introducing the  Gustafson-Kessel Fuzzy Mapper Graphs algorithm, which replaces the spherical covers by ellipsoidal cover, which is useful for high dimensional datasets because real world datasets are not always symmetrical or spherical. We then develop the stability framework for the graphs
produced by Gustafson Kessel Mapper graph and Shape Fuzzy $C$-Mean Mapper graph.
We prove that the memberships depend smoothly on the fuzzifier,
establish a precise condition for the existence of edges, and
show that the graph is locally stable under small perturbations.
We describe the critical event structure of graph changes in
terms of threshold crossings of the membership functions and
show that the graph is constant between consecutive critical
events. When the threshold-crossing set is finite, this yields
an eventual freezing threshold.
Finally, we show empirically that the Gustafson Kessel Mapper is more
stable than the Shape Fuzzy C Means for high-dimensional complex datasets.
\end{abstract}


\end{frontmatter}

\section{Introduction}
\label{sec:intro}

The Mapper algorithm~\cite{singh2007} has become one of the tools in Topological Data Analysis for
analysing the shape of complex, high-dimensional data. It
transforms datasets into a graph, which is a 2D/3D representation of the datasets. It provides the summaries
of connectivity, loops, and hidden geometric relationships that
sometimes ML methods miss~\cite {chazal2021,carriere2017}.
It has been applied in various fields like bioinformatics, neuroscience,
social networks, and many more. Some applications
can be found in~\cite{hasegan2024,li2015,madukpe2025,nicolau2011,
rafique2020,zhou2021}.
Since it has so many applications, selecting its appropriate parameters is still a challenging task.
Particularly, the choice of filter function, cover resolution,
and overlap parameters highly influences the resulting
graph.
Several works have addressed this issue by studying stability
conditions and robustness properties of Mapper
graphs~\cite{ben2006,carriere2017,carriere2018,dey2016}.

The conventional Mapper~\cite{singh2007} uses hard partitioning and rigid interval covers, but for real-world datasets, boundaries cannot always be rigid, and so
 Fuzzy variants of Mapper have been proposed.
The F-Mapper algorithm~\cite{bui2020fmapper} uses Fuzzy $C$-Mean
(FCM) to generate overlapping covers. It creates soft cover rather than hard ones, which becomes more useful for real-world datasets. But it also contains the same number of parameters; the filter function, the fcm cover that requires the number of clusters and membership factor and the dbscan algorithm. To limit these parameters, The Shape Fuzzy $C$-Mean (SFCM) algorithm~\cite{bui2021sfcm}
combines FCM directly with the Mapper nerve construction. Thus, in this method, we only need the number of clusters and the threshold condition.

Although it solves a lot of our purpose, it again has two major limitations:
The first one is that it uses FCM to generate cover, which  assumes
spherical cluster geometry.
In many real world datasets, biological structures, and medical image clusters are non-spherical, and the Euclidean cover
misrepresents the true cluster boundaries.
Second, the stability of the SFCM graph with respect to the
fuzzifier parameter $m$ has not been studied.
It is typically considered as $m = 2$ without theoretical
justification~\cite{bezdek1981}, and without knowing whether small changes in $m$ can affect the graph structure.

We address both limitations.
We propose the \emph{Gustafson-Kessel Fuzzy Mapper Graphs} (GK Mapper)
algorithm, which replaces the Euclidean FCM cover of SFCM with
a cover generated by the Gustafson-Kessel FCM (GK-FCM)
algorithm~\cite{gustafson1979,bezdek1981}, which considers an ellipsoidal structure for the cluster.
We then develop a stability framework for GK Mapper and empirically show that GK-Mapper often performs well across several aspects for which theoretical foundations are developed in this paper. The
contributions of the paper are as follows:

\begin{enumerate}
    \item We propose the GK-Mapper algorithm, which modifies the cover construction of the  SFCM algorithm with the Gustafson-Kessel-based cover.

   \item We characterise the edgeless-zone boundary by the critical threshold
$t_{\mathrm{crit}}(m)=\max_i\max_{j\neq k}\min\{u_{ij}(m),u_{ik}(m)\}$,
above which the graph becomes edgeless.

  \item We then prove a local structural stability theorem with a computable stability radius $r^*$. This radius indicates how far one can vary a chosen value of $m$, obtaining the same Mapper graph.

    \item We show that the GK-Mapper graph can change only at
          threshold-crossing events and is constant between
          consecutive critical events. We further provide a
          crossing-count bound for the number of critical events
          and recover the estimate $|\mathcal{T}|\le nc$ under a
          single-crossing condition.

    \item We empirically show that GK-Mapper performs well in all these cases as compared to SFCM.
\end{enumerate}

The rest of this paper is organised as follows.
Section~\ref{sec:Background} introduces the necessary background
and definitions.
Section~\ref{sec:safcm} proposes the GK-Mapper algorithm.
Section~\ref{sec:regularity} establishes membership regularity
for both FCM and GK-FCM.
Section~\ref{sec:results} presents the main stability framework.
Section~\ref{sec:empirical_validation} provides empirical
validation on synthetic and real-world datasets.
Section~\ref{sec:discussion} discusses the implications and
limitations of the framework.
Section~\ref{sec:conclusion} concludes the paper.

\section{Background}
\label{sec:Background}

Fuzzy set theory~\cite{zimmermann2011} extends classical set
membership by allowing each element to belong to a set with a
degree in the interval $[0,1]$ rather than in a strictly binary
manner.
It is useful for real-world data, where cluster boundaries are
often vague, overlapping, or uncertain.
In clustering, such partial memberships provide a natural way to
model ambiguity at the interfaces between groups.

The Fuzzy $C$-Mean (FCM) algorithm~\cite{bezdek1984fcm} is one of
the most widely used fuzzy clustering methods. It assigns
memberships to all clusters and determines cluster centres by
minimising a weighted objective function.
This soft partitioning makes FCM more flexible than hard
partitioning clustering methods, especially when the data contains
overlap or gradual transitions between groups.
In the standard formulation, however, FCM relies on the
Euclidean distance, which implicitly favours approximately
spherical clusters.

This section provides a brief overview of the definitions and
notations related to the GK-FCM, Mapper, and SFCM algorithms~\cite{bezdek1981,gustafson1979,singh2007,bui2021sfcm}.

\begin{definition}[Gustafson-Kessel FCM~\cite{gustafson1979,bezdek1981}]
\label{def:gkfcm}
The Gustafson-Kessel FCM (GK-FCM) algorithm addresses the
spherical-cluster limitation of the standard FCM by replacing the
Euclidean metric with a cluster-adaptive Mahalanobis-type
distance~\cite{demaesschalck2000}.
Each cluster is allowed to adapt its shape according to the
local covariance structure of the data, making GK-FCM
particularly suitable for datasets with ellipsoidal or
directionally stretched clusters.
This adaptive geometry provides the foundation for the cover
used in GK-Mapper.
Further detail on clustering variants and their applications
can be found in~\cite{bezdek1981,krishnapuram1993,xing2002,zhang2004}.
\end{definition}

\begin{algorithm}[t]
\caption{Gustafson-Kessel Fuzzy $C$-Mean (GK-FCM)}
\label{alg:gkfcm}
\begin{algorithmic}[1]
\Require Dataset $X = \{x_1, \dots, x_n\} \subset \mathbb{R}^p$,
         number of clusters $c$, fuzzifier $m > 1$, volume
         parameters $q_j > 0$, tolerance $\mathrm{tol}$
\Ensure  Membership matrix $U = [u_{ij}]$,
         cluster centres $V = \{v_1, \dots, v_c\}$

\State Initialise $U^{(0)}$ with $\sum_{j=1}^{c} u_{ij} = 1$ for
       all $i$; set $t \gets 0$

\Repeat
    \For{$j = 1$ to $c$}
        \State $v_j^{(t)} =
               \dfrac{\sum_{i=1}^{n} (u_{ij}^{(t)})^m x_i}
                     {\sum_{i=1}^{n} (u_{ij}^{(t)})^m}$
    \EndFor
    \For{$j = 1$ to $c$}
        \State $S_{f_j}^{(t)} =
               \sum_{i=1}^{n}(u_{ij}^{(t)})^m
               (x_i - v_j^{(t)})(x_i - v_j^{(t)})^\top$
        \State $A_j^{(t)} =
               q_j \cdot \det(S_{f_j}^{(t)})^{1/p}
               \cdot (S_{f_j}^{(t)})^{-1}$
    \EndFor
    \For{$i = 1$ to $n$; $j = 1$ to $c$}
        \State $d_{ij}^{(t)} =
               \sqrt{(x_i - v_j^{(t)})^\top A_j^{(t)}
                     (x_i - v_j^{(t)})}$
    \EndFor
    \For{$i = 1$ to $n$; $j = 1$ to $c$}
        \State $u_{ij}^{(t+1)} =
               \!\left[
               \sum_{k=1}^{c}
               \!\left(\dfrac{d_{ij}^{(t)}}{d_{ik}^{(t)}}\right)^{\!2/(m-1)}
               \right]^{-1}$
    \EndFor
    \State $t \gets t + 1$
\Until{$\|U^{(t)} - U^{(t-1)}\| < \mathrm{tol}$}
\State \Return $U,\, V$
\end{algorithmic}
\end{algorithm}

\begin{definition}[Mapper Algorithm~\cite{singh2007}]
\label{def:mapper}
Let $X$ be a dataset, let $f : X \to \mathbb{R}^d$ be a
continuous filter function, and let $\mathcal{U} = \{U_a\}$
be an open cover of $f(X)$.
For each $U_a \in \mathcal{U}$, apply a clustering
algorithm to the preimage $f^{-1}(U_a)$, producing
clusters $\{C_{a,1}, C_{a,2}, \ldots\}$.
The \emph{Mapper complex} is the simplicial complex where each
cluster $C_{a, i}$ is a node and two nodes are connected by
an edge whenever $C_{a, i} \cap C_{b,j} \neq \emptyset$.
More generally, a $k$-simplex is added whenever $k{+}1$ clusters
have a common nonempty intersection.
\end{definition}

\begin{definition}[F-Mapper Algorithm~\cite{bui2020fmapper}]
\label{def:fmapper}
Let $X$ be a finite dataset in a metric space, and let
$f : X \to \mathbb{R}$ be a continuous filter.
F-Mapper partitions $f(X)$ into $N$ fuzzy clusters by FCM,
producing membership degrees $u_{ij} \in [0,1]$ satisfying
$\sum_{j=1}^N u_{ij} = 1$.
For a threshold $t \in [0,1]$, define fuzzy cover intervals
$U_j = \{f(x_i) : u_{ij} \geq t\}$ and pullback sets
$f^{-1}(U_j)$.
Each pullback set is clustered into connected components, and
the \emph{F-Mapper complex} is the nerve of these components.
\end{definition}

\begin{definition}[SFCM Algorithm~\cite{bui2021sfcm}]
\label{def:sfcm}
Let $X = \{x_1, \ldots, x_n\} \subset \mathbb{R}^p$,
$c \geq 2$, $m > 1$, and $t \in (T_0, T_1]$, where
$T_0 = \min_{i,j} u_{ij}$ and $T_1 = \min_i \max_j u_{ij}$.
SFCM minimises the FCM objective
$J_m(U,V) = \sum_{i=1}^n \sum_{j=1}^c (u_{ij})^m
\|x_i - v_j\|^2$,
producing clusters $C_j(t) = \{x_i : u_{ij}(m) \geq t\}$.
The \emph{SFCM graph} is $G_t(m) = (V, E(m))$ where
$V = \{1,\ldots,c\}$ and
$E(m) = \{(j,k) : j \neq k,\; C_j \cap C_k \neq \emptyset\}$.
The complete algorithm is given in Algorithm~\ref{alg:sfcm}.
\end{definition}

\begin{algorithm}[t]
\caption{Shape Fuzzy $C$-Mean (SFCM)}
\label{alg:sfcm}
\begin{algorithmic}[1]
\Require $X = \{x_1, \dots, x_n\} \subset \mathbb{R}^p$,
         number of clusters $c$, fuzzifier $m > 1$,
         overlap threshold $t$,
         tolerance $\mathrm{tol} \in (0,1)$,
         maximum iterations $k_{\max}$
\Ensure  SFCM graph $G_t = (V, E_t)$

\State Initialise $U^{(0)} = [u_{ij}^{(0)}]$ such that
       $u_{ij}^{(0)} \in [0,1]$ and
       $\sum_{j=1}^{c} u_{ij}^{(0)} = 1$ for all $i$
\State Set $k \gets 0$

\Repeat
    \For{$j = 1$ to $c$}
        \State $v_j \gets \dfrac{\displaystyle\sum_{i=1}^{n}
               \bigl(u_{ij}^{(k)}\bigr)^m x_i}
               {\displaystyle\sum_{i=1}^{n}
               \bigl(u_{ij}^{(k)}\bigr)^m}$
    \EndFor
    \For{$i = 1$ to $n$; $j = 1$ to $c$}
        \State $u_{ij}^{(k+1)} \gets
               \left[
               \displaystyle\sum_{\ell=1}^{c}
               \left(
               \frac{\|x_i - v_j\|}{\|x_i - v_\ell\|}
               \right)^{\!2/(m-1)}
               \right]^{-1}$
    \EndFor
    \State $k \gets k + 1$
\Until{$\max_{i,j}\,\bigl|u_{ij}^{(k)} - u_{ij}^{(k-1)}\bigr| < \mathrm{tol}$
       \textbf{ or } $k = k_{\max}$}

\State Compute $T_0 \gets \min_{i,j}\; u_{ij}$ \quad and \quad
       $T_1 \gets \min_{i}\,\max_{j}\; u_{ij}$
\State Clamp $t \gets \max(T_0,\,\min(t,\, T_1))$

\For{$j = 1$ to $c$}
    \State $C_j(t) \gets \{x_i \in X : u_{ij} \geq t\}$
\EndFor

\State $V \gets \{1, \dots, c\}$, \quad $E_t \gets \emptyset$
\For{$j = 1$ to $c-1$; \; $k = j{+}1$ to $c$}
    \If{$C_j(t) \cap C_k(t) \neq \emptyset$}
        \State $E_t \gets E_t \cup \{(j,\, k)\}$
    \EndIf
\EndFor

\State \Return $G_t = (V,\, E_t)$
\end{algorithmic}
\end{algorithm}

\section{GK-Fuzzy Mapper Algorithm}
\label{sec:safcm}

The SFCM algorithm uses the Euclidean-distance assumption
from FCM, which naturally favours spherical cluster shapes.
However, many real-world datasets contain clusters that are
elongated, ellipsoidal, or otherwise non-spherical.
To address this limitation, we introduce the Gustafson Kessel Mapper(GK-Mapper) algorithm.
This method removes the spherical constraint by replacing the
Euclidean cover used in SFCM with a geometry-adaptive cover
derived from the Gustafson-Kessel FCM
algorithm~\cite{gustafson1979}.
Compared with SFCM, GK-Mapper uses the same number of parameters.
The only change is that Euclidean distance is replaced by a
cluster-adaptive distance~\cite[Theorem~22.1]{bezdek1981}.

The detailed computational procedure is presented in
Algorithm~\ref{alg:safcm}.
The algorithm follows the standard fuzzy clustering framework,
where cluster centres and memberships are iteratively updated
using the Gustafson-Kessel adaptive distance.
After convergence, the fuzzy memberships are thresholded to
construct the adaptive cover, and the Mapper graph is obtained
by connecting clusters with nonempty intersections.

\begin{algorithm}[t]
\caption{Gustafson Kessel Mapper (GK-Mapper)}
\label{alg:safcm}
\begin{algorithmic}[1]
\Require $X=\{x_1,\dots,x_n\}\subset \mathbb{R}^p$,
         number of clusters $c$, fuzzifier $m>1$,
         threshold $t$, tolerance $\mathrm{tol}$
\Ensure  GK-Mapper graph $G_t(m)=(V,E)$

\State Initialise $U^{(0)}$ with
       $\sum_{j=1}^{c}u_{ij}^{(0)}=1$ for all $i$
\Repeat
    \For{$j=1$ to $c$}
        \State $v_j=\dfrac{\sum_{i=1}^{n}(u_{ij})^m x_i}
                          {\sum_{i=1}^{n}(u_{ij})^m}$
        \State $S_{f_j}=\sum_{i=1}^{n}(u_{ij})^m
               (x_i-v_j)(x_i-v_j)^\top$
        \State $A_j=q_j\,\det(S_{f_j})^{1/p}S_{f_j}^{-1}$
    \EndFor
    \For{$i=1$ to $n$; $j=1$ to $c$}
        \State $d_{ij}^{\mathrm{GK}}=
               \sqrt{(x_i-v_j)^\top A_j(x_i-v_j)}$
    \EndFor
    \For{$i=1$ to $n$; $j=1$ to $c$}
        \State $u_{ij}=
               \!\left[
               \sum_{k=1}^{c}
               \!\left(\dfrac{d_{ij}^{\mathrm{GK}}}
                             {d_{ik}^{\mathrm{GK}}}\right)^{\!2/(m-1)}
               \right]^{-1}$
    \EndFor
\Until{$\|U^{(t+1)}-U^{(t)}\|_F<\mathrm{tol}$}
\For{$j=1$ to $c$}
    \State $C_j(t)\gets\{x_i\in X: u_{ij}\ge t\}$
\EndFor
\State $V\gets \{1,\dots,c\}$, $E\gets \emptyset$
\For{$j=1$ to $c-1$; $k=j{+}1$ to $c$}
    \If{$C_j(t)\cap C_k(t)\neq \emptyset$}
        $E\gets E\cup \{(j,k)\}$
    \EndIf
\EndFor
\State \Return $G_t(m)=(V,E)$
\end{algorithmic}
\end{algorithm}


\section{Membership Regularity}
\label{sec:regularity}

Before developing the stability framework, we establish that
the GK-FCM membership function depends smoothly on the fuzzifier
parameter $m$ along the optimisation path.
This regularity is the foundation on which all subsequent
results rest.
The following proposition is stated for the moving-centre
setting, where the centres $v_j(m)$ and (in the GK-Mapper case)
the adaptive matrices $A_j(m)$ depend on $m$. The arguments
depend only on the composition structure of the membership
formula, and not on the specific distance used. 
Before stating the main regularity result, we fix the standing assumptions 
that govern Sections~4 and~5.

\begin{assumption}[H1-Continuity]
\label{ass:H1}
For each $i \in \{1,\ldots,n\}$ and $j \in \{1,\ldots,c\}$, the membership 
function $m \mapsto u_{ij}(m)$ is continuous at $m_0$.
\end{assumption}

\begin{assumption}[H2-$C^1$ Optimisation Path]
\label{ass:H2}
The optimisation path $m \mapsto V(m) = \{v_1(m), \ldots, v_c(m)\}$ is $C^1$ 
on an open interval $\mathcal{I} \ni m_0$. In the GK-Mapper case, the adaptive 
matrices $m \mapsto A_j(m)$ are additionally $C^1$ on $\mathcal{I}$ for every 
$j \in \{1,\ldots,c\}$. Furthermore, the non-degeneracy condition
\begin{equation}\label{eq:nondeg}
    d_{il}(m_0) > 0 \qquad \forall\, i \in \{1,\ldots,n\},\; 
    l \in \{1,\ldots,c\}
\end{equation}
holds, that is, no data point coincides with any cluster centre at $m_0$.
\end{assumption}

\begin{proposition}[Membership Regularity Along the Optimisation Path]
\label{prop:regularity}
Let $m\mapsto V(m)=\{v_1(m),\dots,v_c(m)\}$ be a $C^1$ path of
cluster centres on an interval $\mathcal I\subset (1,\infty)$.
In the GK-Mapper case, also assume that $m\mapsto A_j(m)$ is $C^1$
for every $j\in\{1,\dots,c\}$. Define
\[
    d_{ij}(m) =
    \begin{cases}
        \|x_i-v_j(m)\|, & \text{SFCM},\\[4pt]
        \sqrt{(x_i-v_j(m))^\top A_j(m)(x_i-v_j(m))}, &
        \text{GK},
    \end{cases}
\]
and assume $d_{ij}(m)>0$ for all $i,j$ and all $m\in\mathcal I$.
Set
\[
r_{ijk}(m)=\frac{d_{ij}(m)}{d_{ik}(m)},
\qquad
b(m)=\frac{2}{m-1},
\]
and
\[
D_{ij}(m)=\sum_{k=1}^{c} r_{ijk}(m)^{b(m)}.
\]
Then
\[
u_{ij}(m)=\frac{1}{D_{ij}(m)}
\]
satisfies:
\begin{enumerate}[label=\upshape(\roman*)]
    \item $u_{ij}(m)\in(0,1)$ for all $m\in\mathcal I$;
    \item $u_{ij}(m)$ is of class $C^1$ on $\mathcal I$;
    \item $\sum_{j=1}^{c}u_{ij}(m)=1$ for all $i$ and $m$;
    \item $u_{ij}(m)$ is differentiable with
    \[
        u_{ij}'(m)=T_1(m)+T_2(m),
    \]
    where
    \[
        T_1(m)
        =
        \frac{2u_{ij}(m)^2}{(m-1)^2}
        \sum_{k=1}^{c}
        r_{ijk}(m)^{b(m)}\ln r_{ijk}(m),
    \]
    and $T_2(m)$ contains the contribution coming from the
    motion of the centres and, in the GK-Mapper case, the adaptive
    matrices.
\end{enumerate}
More explicitly, for SFCM,
\[
        T_2^{\mathrm{SFCM}}(m)
        =
        \frac{2u_{ij}(m)^2}{m-1}
        \sum_{l=1}^{c}
        \frac{w_{ijl}}{d_{il}(m)^2}
        (x_i-v_l)^\top \dot v_l,
\]
whereas for GK-Mapper,
\begin{align*}
        T_2^{\mathrm{GK}}(m)
        &=
        \frac{2u_{ij}(m)^2}{m-1}
        \sum_{l=1}^{c}
        \frac{w_{ijl}}{d_{il}(m)^2} \\
        &\quad\times
        \left[
        (x_i-v_l)^\top A_l\dot v_l
        -\frac12 (x_i-v_l)^\top \dot A_l(x_i-v_l)
        \right],
\end{align*}
with
\[
        w_{ijl}=
        \begin{cases}
            D_{ij}(m)-1, & l=j,\\
            -r_{ijl}(m)^{b(m)}, & l\neq j.
        \end{cases}
\]
\end{proposition}

\begin{proof}
Since $d_{ij}(m)>0$ and the distance functions are $C^1$ in $m$
under the stated assumptions, each ratio $r_{ijk}(m)$ is strictly
positive and $C^1$. For $m>1$, $b(m)>0$, so each summand
$r_{ijk}(m)^{b(m)}$ is positive and $D_{ij}(m)>0$. Since the
term $k=j$ equals $1$ and $c\geq 2$, we have $D_{ij}(m)>1$.
Therefore $u_{ij}(m)\in(0,1)$.

Next, each summand can be written as
\[
    r_{ijk}(m)^{b(m)}
    =
    \exp\!\bigl(b(m)\ln r_{ijk}(m)\bigr).
\]
Since both $b(m)$ and $\ln r_{ijk}(m)$ are $C^1$ on
$\mathcal I$, it follows that $D_{ij}$ is $C^1$. Hence
$u_{ij}=D_{ij}^{-1}$ is also $C^1$ on $\mathcal I$.

The identity $\sum_{j=1}^{c}u_{ij}(m)=1$ is the standard
normalisation property of fuzzy
memberships~\cite[Def.~5.1]{bezdek1981}.

Finally, differentiating $u_{ij}=D_{ij}^{-1}$ gives
\[
    u_{ij}'(m)=-u_{ij}(m)^2D_{ij}'(m).
\]
Using $b'(m)=-2/(m-1)^2$ and the chain rule,
\begin{align*}
\frac{d}{dm}\left[r_{ijk}(m)^{b(m)}\right]
&=
r_{ijk}(m)^{b(m)}
\Bigg[
-\frac{2}{(m-1)^2}\ln r_{ijk}(m) \\
&\qquad\qquad
+b(m)\frac{\dot r_{ijk}(m)}{r_{ijk}(m)}
\Bigg].
\end{align*}
Therefore,
\begin{align*}
D_{ij}'(m)
&=
-\frac{2}{(m-1)^2}
\sum_{k=1}^{c}r_{ijk}(m)^{b(m)}\ln r_{ijk}(m)\\
&\quad+
b(m)
\sum_{k=1}^{c}
r_{ijk}(m)^{b(m)}
\frac{\dot r_{ijk}(m)}{r_{ijk}(m)}.
\end{align*}
Substituting this expression into $u_{ij}'=-u_{ij}^2D_{ij}'$
gives the decomposition $u_{ij}'(m)=T_1(m)+T_2(m)$. The stated
forms of $T_2^{\mathrm{SFCM}}$ and $T_2^{\mathrm{GK}}$ follow
by differentiating the corresponding distance functions and
collecting the terms associated with each moving centre and
adaptive matrix.
\end{proof}

\begin{corollary}[Non-monotonicity Along the Optimisation Path]
\label{cor:nonmonotone}
Under the hypotheses of Proposition~\ref{prop:regularity}, each
membership function $u_{ij}(m)$ is $C^1$ on $\mathcal I$.
However, along the optimisation path, $u_{ij}(m)$ is not
necessarily monotone.

Indeed, the derivative has the form
\[
    u_{ij}'(m)=T_1(m)+T_2(m),
\]
where $T_1$ is determined by the fuzzifier-dependent exponent and
distance ratios, while $T_2$ contains the effect of centre motion
and, in GK-Mapper, adaptive-matrix motion. Since $T_2$ may have
either sign, the sign of $u_{ij}'(m)$ is not determined by the
distance ratios alone. Consequently, monotonicity of $u_{ij}(m)$
cannot be assumed without additional restrictions on the
optimisation path.
\end{corollary}

\begin{proof}
The result follows directly from the derivative decomposition in
Proposition~\ref{prop:regularity}. The term $T_2$ depends on
$\dot v_l$ and, in the GK-Mapper case, on $\dot A_l$. These quantities
may vary in direction and magnitude along the optimisation path.
Hence $T_2$ may be positive, negative, or zero, and no general
sign condition for $u_{ij}'(m)$ follows from the membership
formula alone. Therefore $u_{ij}(m)$ need not be monotone on
$\mathcal I$.
\end{proof}

\begin{remark}
\label{rem:universal}
Proposition~\ref{prop:regularity} and
Corollary~\ref{cor:nonmonotone} apply to both SFCM and GK-Mapper.
The only difference is the form of the distance function: SFCM
uses Euclidean distances, while GK-Mapper uses cluster-adaptive
Gustafson-Kessel distances. Hence, the subsequent stability
results apply to both constructions, with the GK-Mapper case
including the additional contribution from the evolution of the
adaptive matrices.
\end{remark}

\section{Main Results}
\label{sec:results}

In this section, we present the main theoretical results that
describe how the GK-Mapper graphs change as the fuzzifier
$m$ varies.
These results identify the parameter setting in which the graph
carries structural information, establishes a local stability
zone around any reference value of $m$, quantify graph variation
under small perturbations, and describe the critical event
structure of the graph along the moving centre optimisation path.

We begin by characterising the Edgeless Zone
(Section~\ref{sub:graveyard}), where the graph has no edges.
We then establish a Stability Zone
(Section~\ref{sub:stability}) in which the graph remains
unchanged under small changes in $m$.
We next analyse the Instability Zone
(Section~\ref{sub:instability}) and derive an upper bound on
edge changes.
Finally, we describe the Critical Event Structure
(Section~\ref{sub:critical}) and analyse when the graph
eventually freezes beyond a finite threshold $m^{**}$.

\subsection{The Edgeless Zone}
\label{sub:graveyard}

A parameter pair $(m, t)$ lies in the \emph{Edgeless Zone}
if $G_t(m)$ has zero edges, then the graph carries no
structural information.
Theorem~\ref{thm:graveyard} establishes the necessary and
sufficient condition for avoiding the Edgeless Zone.
A visual illustration is provided in
Fig.~\ref{fig:theorem1_graveyard_zone}.

\begin{theorem}[Edgeless Zone]
\label{thm:graveyard}
Let $G _t(m)$ be the GK-Mapper graph.
Then $G _t(m)$ has at least one edge if and only if
\begin{equation}
    t
    \;\leq\;
    \max_{i}\,\max_{j \neq k}
    \min\bigl\{u_{ij}(m),\, u_{ik}(m)\bigr\}.
    \label{eq:graveyard-cond}
\end{equation}
Equivalently, $G_t(m)$ has no edges iff 
$t > \max_i \max_{j\neq k}\min\{u_{ij}(m), u_{ik}(m)\}$.
\end{theorem}

\begin{proof}
$(\Rightarrow)$
Suppose $G_t(m)$ contains an edge; then
$C_j \cap C_k \neq \emptyset$ for some $j \neq k$.
Let $x_i$ be a point in this intersection.
Then $u_{ij}(m) \geq t$ and $u_{ik}(m) \geq t$, so
$\min\{u_{ij}(m), u_{ik}(m)\} \geq t$,
giving~\eqref{eq:graveyard-cond}.

$(\Leftarrow)$
If~\eqref{eq:graveyard-cond} holds, there exist $i$, $j$, $k$
with $j \neq k$ such that $u_{ij}(m) \geq t$ and
$u_{ik}(m) \geq t$.
Therefore $x_i \in C_j \cap C_k$ and an edge $(j,k)$ exists.
\end{proof}

We define
$t _{\mathrm{crit}}
 = \max_i \max_{j\neq k}\min\{u_{ij}(m), u_{ik}(m)\}$
as the critical threshold above which the graph enters the
Edgeless Zone.

\begin{figure*}[!t]
    \centering
    \begin{subfigure}[t]{0.48\textwidth}
        \centering
        \includegraphics[width=\textwidth]{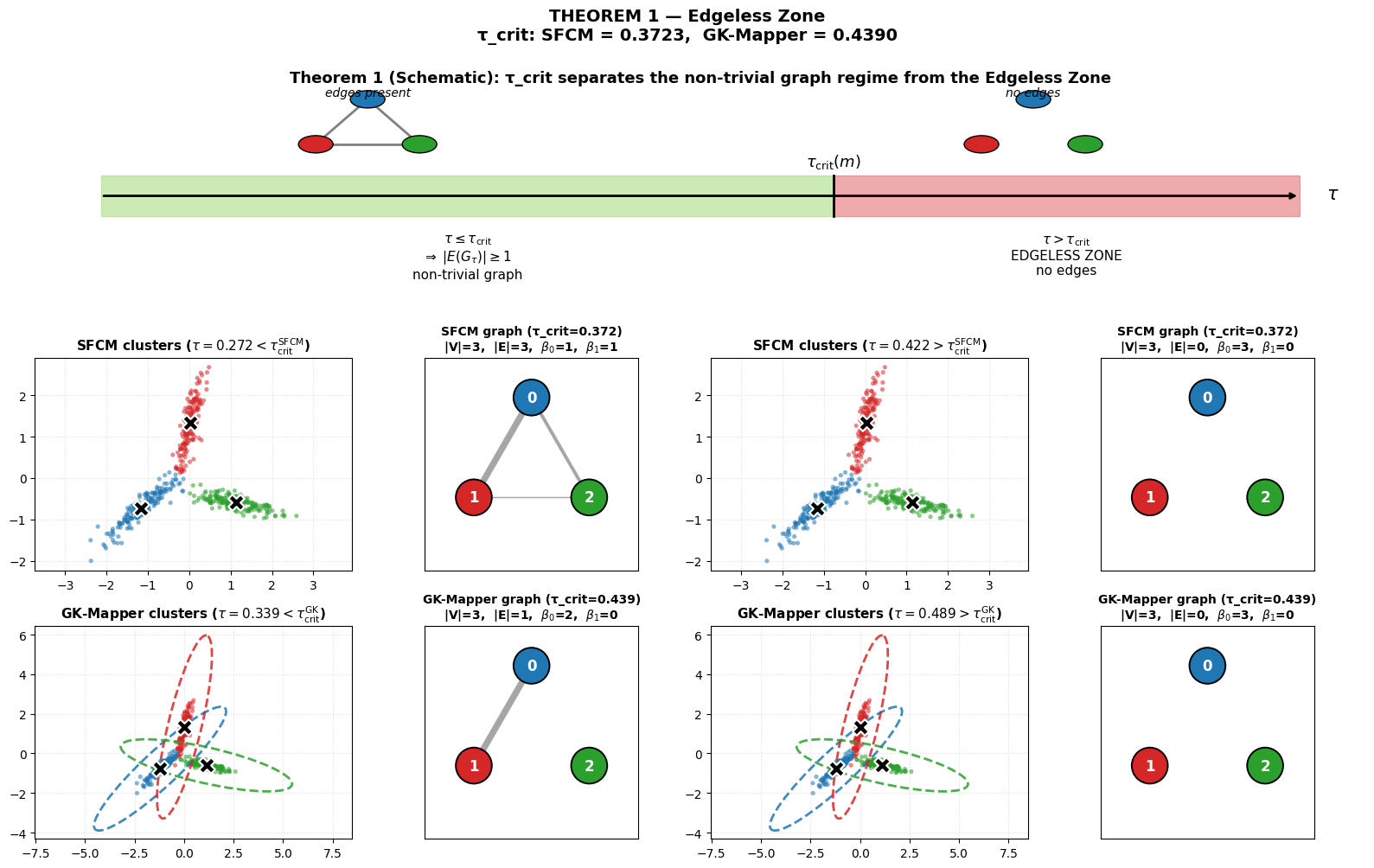}
        \caption{Visual illustration of Theorem~\ref{thm:graveyard}
        for the anisotropic ellipsoidal dataset at $m_0=2.0$.
        The critical threshold $t_{\mathrm{crit}}(m)$ separates
        the non-trivial graph regime from the Edgeless Zone.}
        \label{fig:theorem1_graveyard_zone}
    \end{subfigure}
    \hfill
    \begin{subfigure}[t]{0.48\textwidth}
        \centering
        \includegraphics[width=\textwidth]{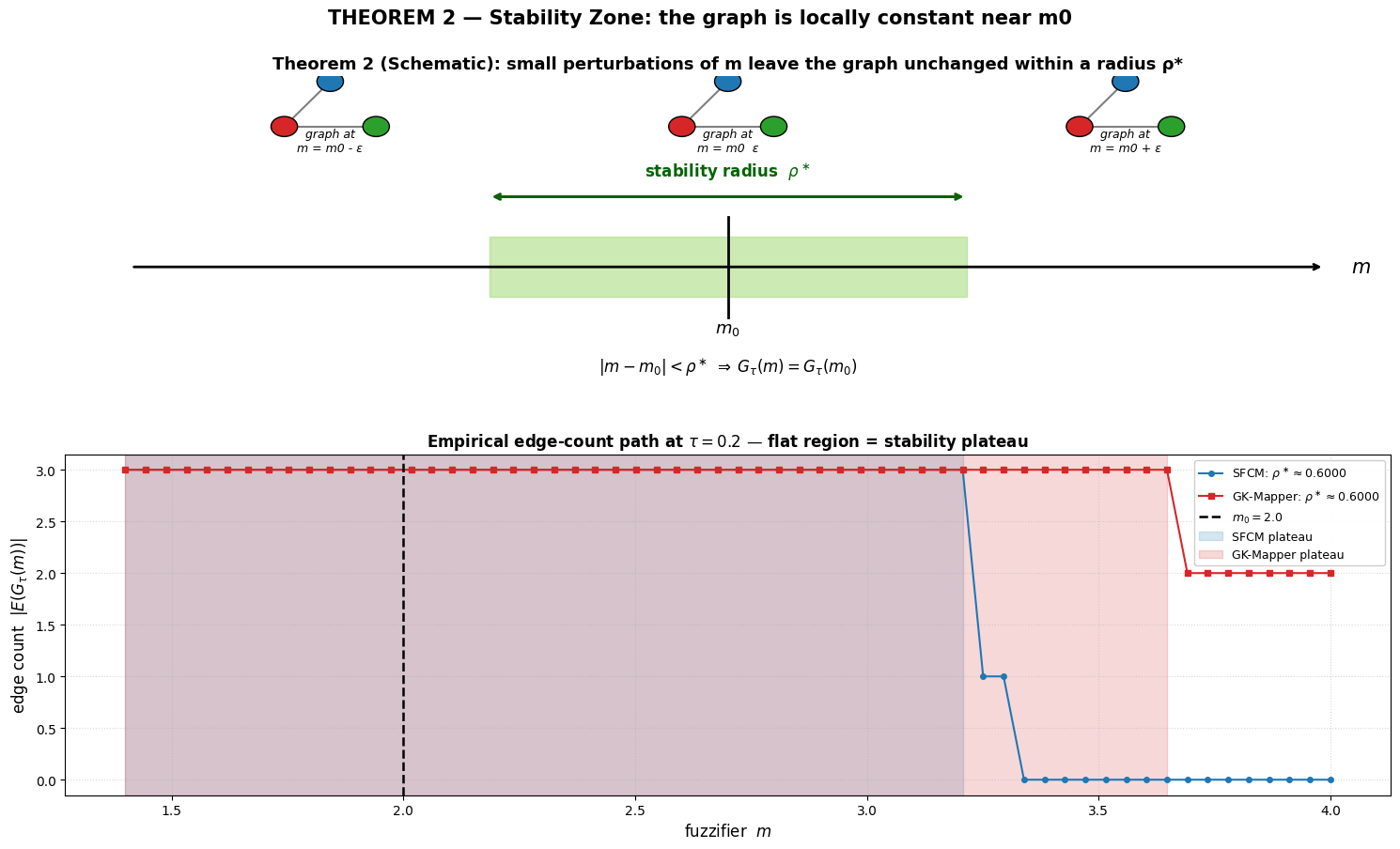}
        \caption{Visual illustration of Theorem~\ref{thm:stability}
        for the anisotropic ellipsoidal dataset at $m_0=2.0$ and
        $t=0.2$. The stability radius $r^{\ast}$ defines a
        local interval where no membership value crosses $t$,
        so $G_{t}(m)=G_{t}(m_0)$. The flat edge-count paths
        confirm the predicted local graph constancy.}
        \label{fig:theorem2_stability_zone}
    \end{subfigure}
    \caption{Empirical illustrations of Theorems~\ref{thm:graveyard}
    and~\ref{thm:stability}. Subfigure~\ref{fig:theorem1_graveyard_zone}
    shows the Edgeless Zone transition, while
    Subfigure~\ref{fig:theorem2_stability_zone} shows local graph
    stability around $m_0=2.0$.}
    \label{fig:theorem1_theorem2_combined}
\end{figure*}

\subsection{The Stability Zone}
\label{sub:stability}

After identifying the region where the graph becomes edgeless , we focus on the area where the graph structure is preserved under small perturbations of the fuzzifier $m$. 
Intuitively, if the membership values do not cross the threshold
$t$, the induced cover and hence the graph topology remain
unchanged.
The following theorem gives a radius $r^*$ for which the graph remains unchanged for a chosen $m$. Its behaviour described is illustrated in
Fig.~\ref{fig:theorem2_stability_zone}.

Throughout Theorem~\ref{thm:stability}, we assume
\begin{equation}
    d_{il}(m_0) > 0,
    \qquad
    \forall\, i \in \{1,\dots,n\},\; l \in \{1,\dots,c\},
    \label{eq:nondeg}
\end{equation}
that is, no data point coincides with any cluster centre at
$m_0$. This ensures that every distance denominator appearing
in the proof is strictly positive. This is the standard
non-degeneracy condition in FCM~\cite{bezdek1981}.

Moreover, for each $i \in \{1,\dots,n\}$ and
$j \in \{1,\dots,c\}$, let $d_{ij}(m)$ denote the distance
from $x_i$ to the cluster centre $v_j(m)$ under the relevant
metric:
\begin{equation}
    d_{ij}(m) =
    \begin{cases}
        \|x_i - v_j(m)\|, & \text{SFCM},\\[4pt]
        \sqrt{(x_i-v_j(m))^\top A_j(m)\,(x_i-v_j(m))}, &
        \text{GK-Mapper}.
    \end{cases}
    \label{eq:dist}
\end{equation}
Define
\[
r_{ijk}(m) = \frac{d_{ij}(m)}{d_{ik}(m)}, \qquad
D_{ij}(m) = \sum_{k=1}^{c} r_{ijk}(m)^{b(m)},
\]
\[
b(m) = \frac{2}{m-1}, \qquad
u_{ij}(m) = \frac{1}{D_{ij}(m)}.
\]
Note that $r_{ijj}(m)\equiv 1$, and therefore
\[
D_{ij}(m)=1+\sum_{k\neq j} r_{ijk}(m)^{b(m)}.
\]

\begin{theorem}[Local Stability Zone]
\label{thm:stability}
Let $m_0>1$ and $t\in(0,1)$ satisfy
\[
    u_{ij}(m_0)\neq t
    \qquad
    \text{for all } i\in\{1,\dots,n\},\ j\in\{1,\dots,c\}.
\]
Assume the non-degeneracy condition~\eqref{eq:nondeg} and
suppose that the hypotheses of Proposition~\ref{prop:regularity}
hold on a neighbourhood of $m_0$. Then there exists $r^*>0$
such that
\begin{equation}
    G_t(m)=G_t(m_0)
    \qquad
    \text{whenever } |m-m_0|<r^*.
    \label{eq:stability}
\end{equation}
Hence, the SFCM and GK-Mapper graphs are locally constant with respect
to the fuzzifier parameter near every non-threshold value $m_0$.
\end{theorem}

\begin{proof}
By Proposition~\ref{prop:regularity}, each membership function
$u_{ij}(m)$ is continuous, indeed $C^1$, in a neighbourhood of
$m_0$. Since $u_{ij}(m_0)\neq t$, define the positive
threshold margin
\[
    d_{ij}^{\mathrm{gap}}
    :=
    |u_{ij}(m_0)-t|>0.
\]
By continuity of $u_{ij}$ at $m_0$, there exists $r_{ij}>0$
such that
\[
    |u_{ij}(m)-u_{ij}(m_0)|<d_{ij}^{\mathrm{gap}}
    \qquad
    \text{whenever } |m-m_0|<r_{ij}.
\]
Since there are only finitely many pairs $(i,j)$, define
\[
    r^*
    :=
    \min_{1\leq i\leq n,\;1\leq j\leq c}r_{ij}.
\]
Then $r^*>0$. Hence, for every $m$ satisfying
$|m-m_0|<r^*$, we have
\[
    |u_{ij}(m)-u_{ij}(m_0)|
    <
    |u_{ij}(m_0)-t|
    \qquad
    \text{for all } i,j.
\]
Therefore $u_{ij}(m)$ and $u_{ij}(m_0)$ lie on the same side of
the threshold $t$. Consequently,
\[
    \mathbf 1[u_{ij}(m)\geq t]
    =
    \mathbf 1[u_{ij}(m_0)\geq t]
    \qquad
    \text{for all } i,j.
\]
Thus the thresholded cluster sets
$C_j(t,m)=\{x_i:u_{ij}(m)\geq t\}$
remain unchanged for all $j$. Since the edge set of
$G_t(m)$ is determined by the nonempty intersections
$C_a(t,m)\cap C_b(t,m)\neq \varnothing$,
the edge set also remains unchanged. Therefore
$G_t(m)=G_t(m_0)$ whenever $|m-m_0|<r^*$.
\end{proof}

\begin{remark}
\label{rem:computable_radius}
The proof above establishes the existence of a local stability radius.
A conservative computable estimate can be obtained from the
derivative formula in Proposition~\ref{prop:regularity}. Let
$J=[m_0-\mathrm{tol},m_0+\mathrm{tol}]$ be a compact neighbourhood
contained in the interval of regularity, and set
$M_{ij}(\mathrm{tol}):=\sup_{s\in J}|u'_{ij}(s)|$.
By Proposition~\ref{prop:regularity}, $u'_{ij}$ is continuous,
so $M_{ij}(\mathrm{tol})<\infty$. The mean value theorem gives
$|u_{ij}(m)-u_{ij}(m_0)|\leq M_{ij}(\mathrm{tol})|m-m_0|$.
Hence, one may take
\[
    r^*
    =
    \min_{i,j}
    \left\{
    \mathrm{tol},\,
    \frac{|u_{ij}(m_0)-t|}{M_{ij}(\mathrm{tol})}
    \right\},
\]
with the convention that if $M_{ij}(\mathrm{tol})=0$, the
corresponding term is taken as $\mathrm{tol}$.
\end{remark}

\begin{remark}
Theorem~\ref{thm:stability} establishes that both GK-Mapper and SFCM
graphs are locally stable near any non-threshold fuzzifier value
$m_0$, with computable stability radii $r^{*,\mathrm{GK}}$
and $r^{*,\mathrm{SFCM}}$. Whether GK-Mapper or SFCM achieves a
larger stability radius depends on the underlying cluster geometry
through the distance ratios,
the rates of evolution of the shape matrices and many other factors. For complex datasets, the
empirical evidence in Section~\ref{sec:empirical_validation}
suggests that GK-Mapper can produce larger stability regions than
SFCM. 
\end{remark}

\subsection{The Instability Zone}
\label{sub:instability}

Having established local stability, we now quantify how many
edges can change when the fuzzifier is perturbed from $m$ to
$m+h$.
For each membership entry, define the threshold indicator
\begin{equation}
\operatorname{ind}_{ij}(m)=\mathbf{1}\!\left[u_{ij}(m)\geq t\right].
\end{equation}
For a pair of clusters $(a,b)$, define the witness count
\begin{equation}
I_{ab}(m)=
\sum_{i=1}^{n}
\operatorname{ind}_{ia}(m)\operatorname{ind}_{ib}(m).
\end{equation}
Thus, $I_{ab}(m)$ counts the number of data points simultaneously
belonging to the thresholded clusters $a$ and $b$. Hence, the
edge $(a,b)$ exists in $G_t(m)$ if and only if $I_{ab}(m)>0$.
This behaviour is illustrated in
Fig.~\ref{fig:theorem3_instability_zone}.

\begin{lemma}
\label{lem:instability}
An edge $(a,b)$ changes between $G_t(m)$ and
$G_t(m+h)$ if and only if
\[
I_{ab}(m)\,I_{ab}(m+h)=0
\quad\text{and}\quad
I_{ab}(m)+I_{ab}(m+h)>0.
\]
Consequently, $G_t(m)\neq G_t(m+h)$ if and only if
the above condition holds for at least one pair $(a,b)$.
\end{lemma}

\begin{proof}
The edge $(a,b)$ exists exactly when $I_{ab}(m)>0$. Therefore,
the edge changes between $m$ and $m+h$ precisely when one
of the two witness counts $I_{ab}(m)$ and $I_{ab}(m+h)$
is positive and the other is zero, which is equivalent to the
stated conditions. The graph changes if and only if at least one
edge changes.
\end{proof}

\begin{theorem}[Edge-Change Bound]
\label{thm:edge-bound}
Define the threshold-crossing set
\[
S_{h}
=
\left\{
(i,j):
\bigl(u_{ij}(m)-t\bigr)
\bigl(u_{ij}(m+h)-t\bigr)<0
\right\}.
\]
For each $i$, let
$K_i=|\{j:(i,j)\in S_{h}\}|$.
Then the number of edge changes satisfies
\begin{equation}
|E_{\mathrm{chg}}|
\leq
\sum_{i=1}^{n}
\left[
\binom{K_i}{2}+K_i(c-K_i)
\right]
\leq
(c-1)|S_{h}|.
\label{eq:edge-bound}
\end{equation}
\end{theorem}

\begin{proof}
An edge $(a,b)$ can change only if, for some data point $x_i$,
at least one of the indicators $\operatorname{ind}_{ia}$ or $\operatorname{ind}_{ib}$ changes
between $m$ and $m+h$. Such a change can occur only when
$(i,a)\in S_{h}$ or $(i,b)\in S_{h}$.

Fix a data point $x_i$ and suppose that $K_i$ of its membership
entries cross the threshold. The affected cluster pairs are of
two types. First, both indices may belong to the crossing set,
giving at most $\binom{K_i}{2}$ pairs. Second, exactly one index
may belong to the crossing set, giving at most $K_i(c-K_i)$
pairs. Hence the number of edge pairs affected by $x_i$ is at
most $\binom{K_i}{2}+K_i(c-K_i)$.
Summing over all data points gives the first inequality.
Since $\binom{K_i}{2}+K_i(c-K_i)\leq K_i(c-1)$, summing yields
$(c-1)\sum_i K_i = (c-1)|S_{h}|$.
\end{proof}

\begin{figure*}[!t]
    \centering
    \begin{subfigure}[t]{0.48\textwidth}
        \centering
        \includegraphics[width=\textwidth]{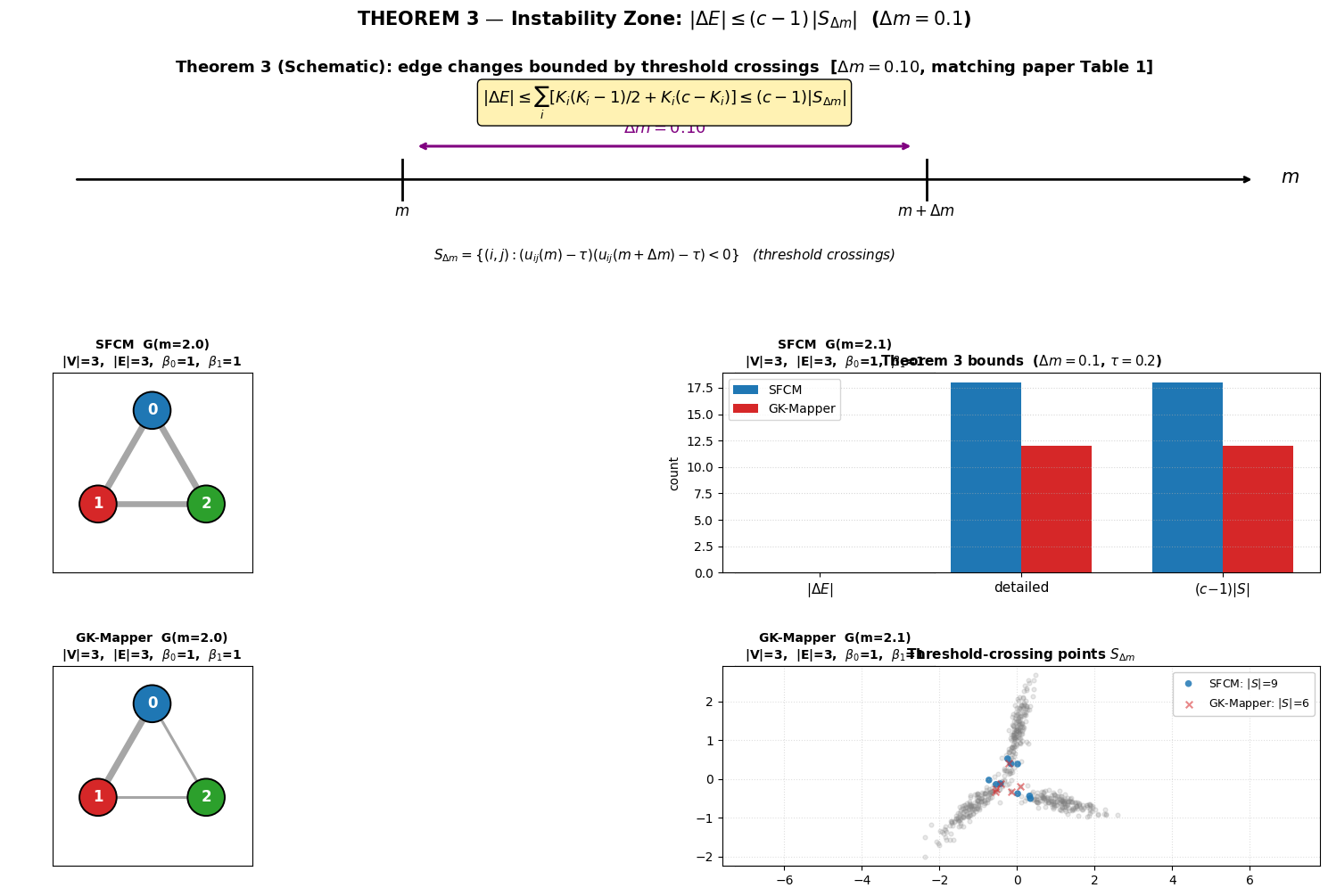}
        \caption{Visual illustration of
        Theorem~\ref{thm:edge-bound} for the anisotropic
        ellipsoidal dataset. Edge changes in the $t$-superlevel
        cluster co-occurrence graph are controlled by membership
        threshold crossings between $m=2.0$ and $m=2.2$. The
        SFCM and GK-Mapper graphs remain unchanged, confirming that
        the theorem gives a conservative upper bound on graph
        instability.}
        \label{fig:theorem3_instability_zone}
    \end{subfigure}
    \hfill
    \begin{subfigure}[t]{0.48\textwidth}
        \centering
        \includegraphics[width=\textwidth]{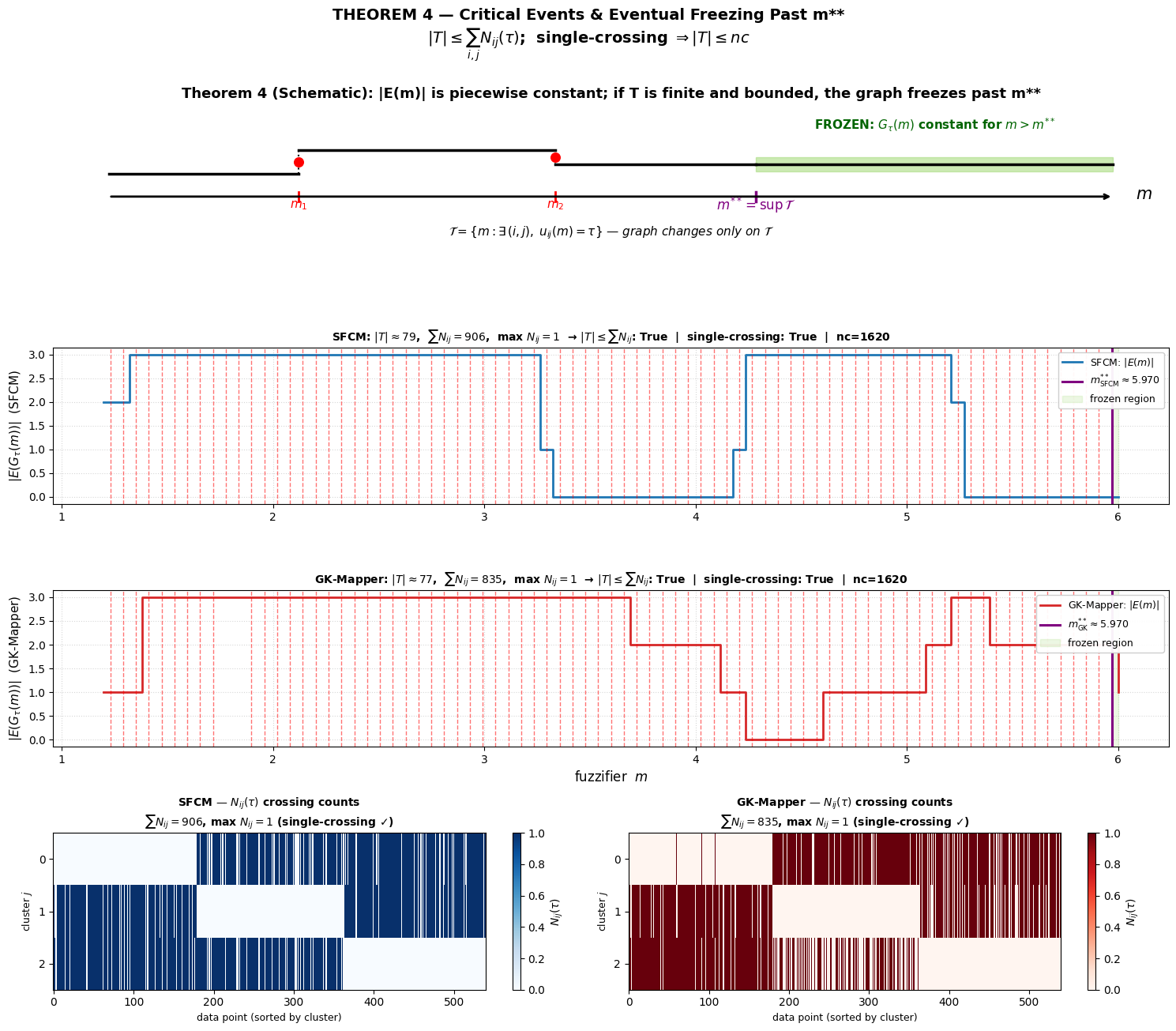}
        \caption{Visual illustration of
        Theorem~\ref{thm:freezing} for the $t$-superlevel
        cluster co-occurrence graph. Graph changes occur only at
        critical fuzzifier values satisfying $u_{ij}(m)=t$.
        The edge-count paths are piecewise constant, with the
        frozen regime appearing after $m^{\ast\ast}$.}
        \label{fig:theorem4_critical_events_freezing}
    \end{subfigure}
    \caption{Empirical illustrations of
    Theorems~\ref{thm:edge-bound} and~\ref{thm:freezing}.
    Subfigure~\ref{fig:theorem3_instability_zone} shows the
    conservative instability bound through membership threshold
    crossings, while
    Subfigure~\ref{fig:theorem4_critical_events_freezing} shows
    critical-event structure and eventual graph freezing.}
    \label{fig:theorem3_theorem4_combined}
\end{figure*}

\subsection{Critical Events and Eventual Freezing}
\label{sub:critical}

Having shown that edge changes are controlled by membership
threshold crossings, we now describe the critical-event structure
of the graph as the fuzzifier $m$ varies. Since the memberships
are evaluated along the optimisation path, they need not be
monotone in $m$. Therefore, a membership value may cross the
threshold $t$ more than once. We formulate the result in
terms of the actual threshold-crossing events.

Let $\mathcal I\subset(1,\infty)$ be the interval of fuzzifier
values under consideration. Define the critical-event set
\[
\mathcal T
=
\left\{
m\in\mathcal I:
\exists\,(i,j)\ \text{such that}\ u_{ij}(m)=t
\right\}.
\]
For each pair $(i,j)$, define the threshold-crossing count
\[
N_{ij}(t)
=
\#\left\{
m\in\mathcal I:u_{ij}(m)=t
\right\}.
\]
Thus, $N_{ij}(t)$ records the number of times the membership
of $x_i$ in cluster $j$ reaches the threshold $t$ on
$\mathcal I$.

\begin{theorem}[Critical Events and Eventual Freezing]
\label{thm:freezing}
Assume that each membership function $u_{ij}(m)$ is continuous
on $\mathcal I$. Then the following statements hold.
\begin{enumerate}[label=\upshape(\roman*)]
    \item $G_t(m)$ can change only at values
    $m\in\mathcal T$. Equivalently, $G_t(m)$ is constant on
    every connected component of $\mathcal I\setminus\mathcal T$.

    \item If $N_{ij}(t)<\infty$ for every pair $(i,j)$, then
    \[
        |\mathcal T|
        \leq
        \sum_{i=1}^{n}\sum_{j=1}^{c}N_{ij}(t).
    \]

    \item If $N_{ij}(t)\leq 1$ for all $(i,j)$, then
    $|\mathcal T|\leq nc$.

    \item If $\mathcal T$ is finite and bounded above in
    $\mathcal I$, and if $m^{**}:=\sup\mathcal T$,
    then $G_t(m)$ is constant on
    $\mathcal I\cap(m^{**},\infty)$.
\end{enumerate}
\end{theorem}

\begin{proof}
For each $i,j$, define the threshold indicator
$\operatorname{ind}_{ij}(m)=\mathbf 1[u_{ij}(m)\geq t]$, so that
$x_i\in C_j(t,m)\Longleftrightarrow\operatorname{ind}_{ij}(m)=1$.
An edge $(a,b)$ exists in $G_t(m)$ if and only if there
exists some $x_i$ such that $\operatorname{ind}_{ia}(m)\operatorname{ind}_{ib}(m)=1$.

\textup{(i)} Suppose $m_1$ and $m_2$ lie in the same connected
component of $\mathcal I\setminus\mathcal T$. Then
$u_{ij}(m)\neq$ throughout the interval between $m_1$ and
$m_2$. Since $u_{ij}$ is continuous, it cannot move from one
side of $t$ to the other without attaining the value $t$.
Hence $\operatorname{ind}_{ij}(m_1)=\operatorname{ind}_{ij}(m_2)$ for all $i,j$, so the
edge set of $G_t(m)$ is unchanged. Consequently, graph
changes can occur only at values in $\mathcal T$.

\textup{(ii)} For each fixed pair $(i,j)$, the equation
$u_{ij}(m)=t$ has exactly $N_{ij}(t)$ solutions on
$\mathcal I$. Hence
$|\mathcal T|\leq\sum_{i=1}^{n}\sum_{j=1}^{c}N_{ij}(t)$,
with inequality because several memberships may cross $t$ at
the same value of $m$.

\textup{(iii)} If $N_{ij}(t)\leq 1$ for all $(i,j)$, then
$\sum_{i,j}N_{ij}(t)\leq nc$, giving $|\mathcal T|\leq nc$.

\textup{(iv)} Since $\mathcal T$ is finite, no critical event
occurs in $\mathcal I\cap(m^{**},\infty)$. By part~(i), the
graph is constant on every connected component of
$\mathcal I\setminus\mathcal T$. Therefore $G_t(m)$ is
constant on $\mathcal I\cap(m^{**},\infty)$.
\end{proof}

\begin{remark}
\label{rem:freezing_interpretation}
Theorem~\ref{thm:freezing} does not assume that the memberships
are monotone in $m$. It only requires continuity. The bound in
part~(ii) uses the actual number of threshold-crossing events,
while part~(iii) gives the simpler estimate $|\mathcal T|\leq nc$
only when each membership reaches the threshold at most once on
$\mathcal I$. Thus, the eventual freezing point $m^{**}$ is
conditional on the critical-event set being finite and bounded
above.
\end{remark}

\section{Empirical Validation}
\label{sec:empirical_validation}

We evaluate GK-Mapper against SFCM on five datasets-Circle,
Anisotropic Ellipsoidal, Stanford Bunny, UCI Handwritten Digits,
and Wisconsin Breast Cancer to test the stability framework of
Section~\ref{sec:results}. The reported quantities are the
critical threshold $t_{\mathrm{crit}}$
(Theorem~\ref{thm:graveyard}; largest $t$ admitting at least
one edge), the empirical stability radius $r^*$
(Theorem~\ref{thm:stability}; local robustness in $m$), the
edge instability $|E_{\mathrm{chg}}|$ under perturbation $h$,
the Theorem~\ref{thm:edge-bound} detailed/simple bounds, and
standard clustering metrics (Silhouette, ARI, matching score).
All reference fuzzifiers $m_0$ are selected automatically from
the search grid. Aggregate results appear in
Table~\ref{tab:all_dataset_results}.

\subsection{Circle Dataset}
\label{subsec:circle}

A synthetic benchmark of 150 points on the unit circle in
$\mathbb{R}^2$ with Gaussian noise ($s=0.05$) and eight
angular sectors as ground-truth labels. Both methods use $c=8$,
$t=0.30$, with $m_0=4.457$ (SFCM) and $m_0=2.286$ (GK-Mapper).
GK-Mapper raises $t_{\mathrm{crit}}$ from $0.4016$ to $0.4652$
and increases $r^*$ from $0.0229$ to $0.0833$. Under $h=0.10$,
both methods yield $|E_{\mathrm{chg}}|=0$, indicating that the
graph structure is unchanged under the selected perturbation.

Clustering metrics show a mixed but favourable trend for GK-Mapper.
SFCM has a slightly higher Silhouette score ($0.495$ vs.\ $0.480$),
whereas GK-Mapper gives a substantially higher ARI ($0.814$ vs.\
$0.513$) and matching score ($0.913$ vs.\ $0.713$). Both methods
recover the expected circular topology, with $|V|=8$, $|E|=8$,
and $b_0=b_1=1$; see Fig.~\ref{fig:unit_circle_combined}. Thus,
on this dataset, GK-Mapper preserves the same topological structure
as SFCM while improving the empirical stability radius and external
label agreement.

\begin{figure*}[!t]
    \centering
    \begin{subfigure}[t]{0.48\textwidth}
        \centering
        \includegraphics[width=\textwidth]{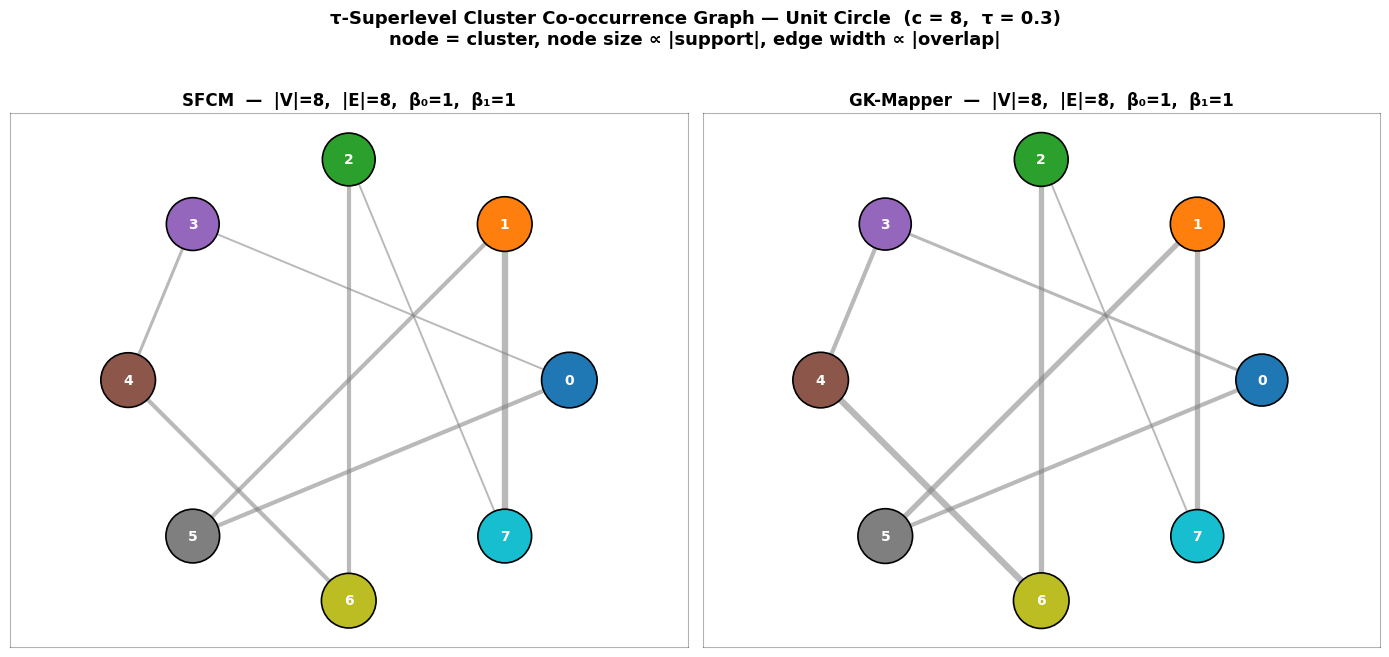}
        \caption{Co-occurrence graphs ($c=8$, $t=0.3$); both
        yield $|V|=|E|=8$, $b_0=b_1=1$.}
        \label{fig:unit_circle_mapper_graph}
    \end{subfigure}
    \hfill
    \begin{subfigure}[t]{0.48\textwidth}
        \centering
        \includegraphics[width=\textwidth]{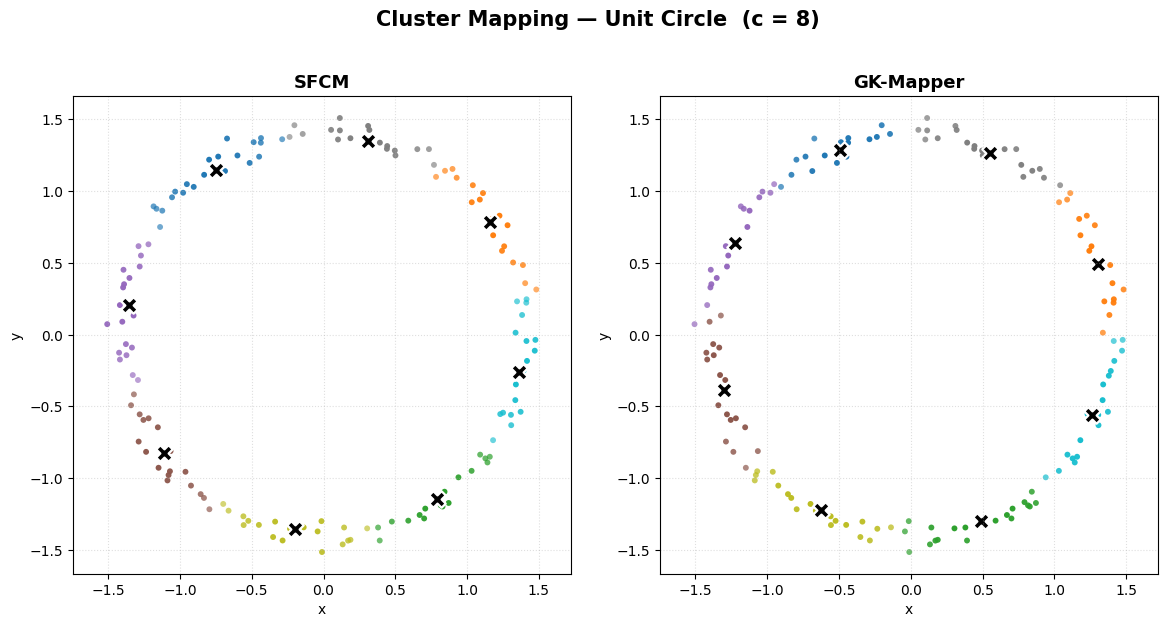}
        \caption{Cluster assignments. Both methods give comparable
        decompositions on this isotropic geometry.}
        \label{fig:unit_circle_cluster_mapping}
    \end{subfigure}
    \caption{Unit circle dataset: SFCM (left) and GK-Mapper (right).}
    \label{fig:unit_circle_combined}
\end{figure*}

\subsection{Anisotropic Ellipsoidal Dataset}
\label{subsec:anisotropic}

Three Gaussian clusters (180 points each) in $\mathbb{R}^2$
with aspect ratios $5{:}1$, $8{:}1$, $4{:}1$, deliberately
violating spherical-cluster assumptions~\cite{riani2015}.
Both methods use $c=3$, $t=0.20$, and $m_0=1.200$.
GK-Mapper improves $t_{\mathrm{crit}}$ from $0.3723$ to $0.4390$
and increases $r^*$ from $0.0132$ to $0.0333$. Under the selected
perturbation $h=0.10$, both methods give
$|E_{\mathrm{chg}}|=0$, so no edge change is observed.

Both methods produce the same graph-level topological summary:
$|V|=3$, $|E|=3$, $b_0=1$, and $b_1=1$
(Fig.~\ref{fig:ellipsoidal_combined}). The clustering metrics are
also very close. SFCM has a slightly higher Silhouette score
($0.684$ vs.\ $0.677$), while both methods obtain the same ARI
($0.962$) and matching score ($0.987$). Hence, on this elongated
dataset, GK-Mapper mainly improves the nontrivial threshold range
and the empirical stability radius while preserving essentially the
same clustering quality as SFCM.

\begin{figure*}[!t]
    \centering
    \begin{subfigure}[t]{0.48\textwidth}
        \centering
        \includegraphics[width=\textwidth]{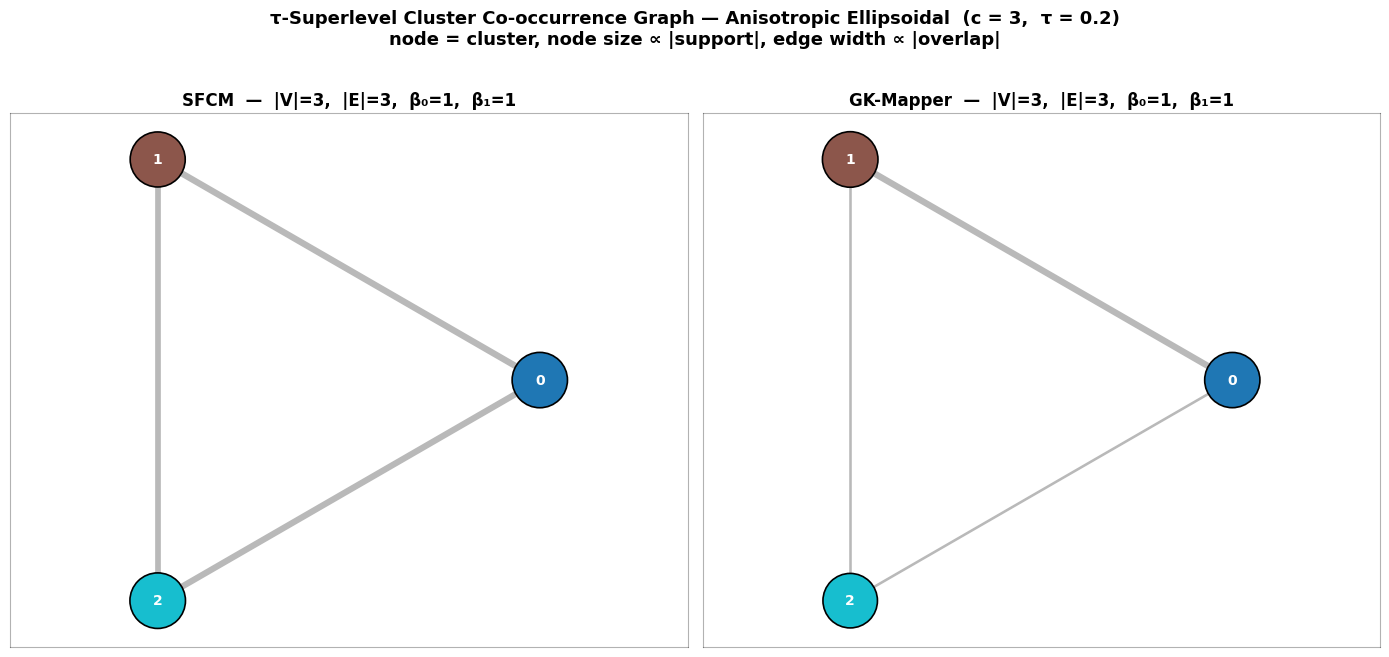}
        \caption{Co-occurrence graphs ($c=3$, $t=0.20$); both
        give $|V|=|E|=3$, $b_0=b_1=1$, but with different
        overlap strengths.}
        \label{fig:ellipsoidal_mapper_graph}
    \end{subfigure}
    \hfill
    \begin{subfigure}[t]{0.48\textwidth}
        \centering
        \includegraphics[width=\textwidth]{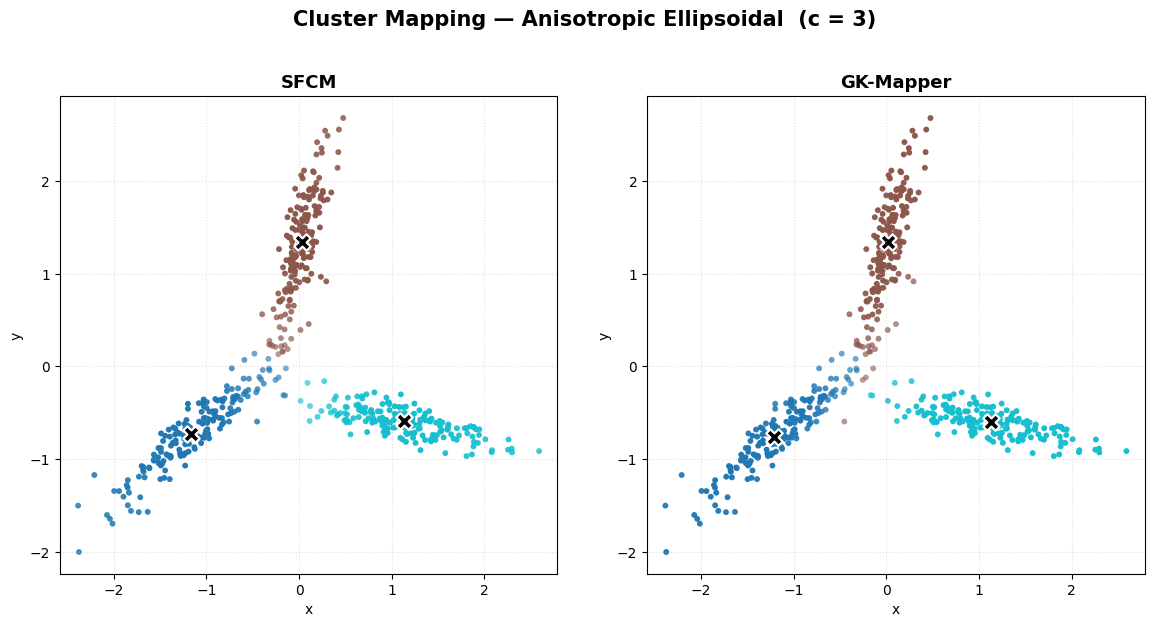}
        \caption{Cluster assignments. GK-Mapper adapts to elongated
        geometries via the Gustafson-Kessel metric.}
        \label{fig:ellipsoidal_cluster_mapping}
    \end{subfigure}
    \caption{Anisotropic ellipsoidal dataset.}
    \label{fig:ellipsoidal_combined}
\end{figure*}

\subsection{Stanford Bunny Dataset}
\label{subsec:bunny}

A 3D point cloud of 5000 points sampled from the Stanford Bunny
mesh~\cite{vanveen2019} (Open3D), centered, rotated, and $\ell_2$
normalized. Both methods use $c=8$, $t=0.25$, with
$m_0=5.000$ (SFCM) and $m_0=3.943$ (GK-Mapper). GK-Mapper raises
$t_{\mathrm{crit}}$ from $0.3762$ to $0.4511$ and slightly improves
$r^*$ from $0.0072$ to $0.0096$. Under $h=0.10$, SFCM shows one
edge change, whereas GK-Mapper shows no edge change.

The GK-Mapper graph is topologically richer than the SFCM graph.
SFCM gives $|V|=8$, $|E|=10$, $b_0=1$, and $b_1=3$, whereas
GK-Mapper gives $|V|=8$, $|E|=13$, $b_0=1$, and $b_1=6$; see
Fig.~\ref{fig:bunny_combined}. The Silhouette score is higher for
SFCM ($0.346$ vs.\ $0.307$), while ARI and matching score are not
available because the point cloud has no class labels. These results
suggest that GK-Mapper retains a more connected and cycle-rich graph
on the Bunny point cloud while also reducing the observed edge
variation under the selected perturbation.

\begin{figure*}[!t]
    \centering
    \begin{subfigure}[t]{0.48\textwidth}
        \centering
        \includegraphics[width=\textwidth]{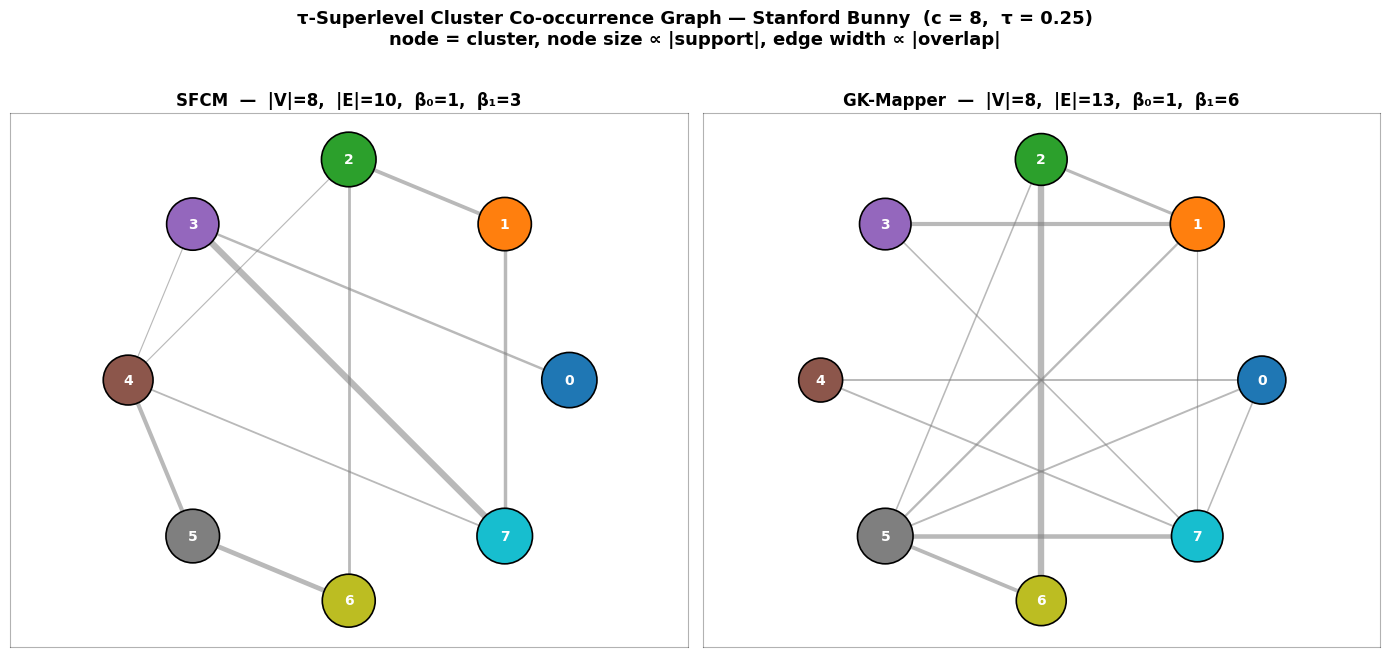}
        \caption{Co-occurrence graphs ($c=8$, $t=0.25$).
        SFCM: $|E|=10$, $b_1=3$;
        GK-Mapper: $|E|=13$, $b_1=6$.}
        \label{fig:bunny_mapper_graph}
    \end{subfigure}
    \hfill
    \begin{subfigure}[t]{0.48\textwidth}
        \centering
        \includegraphics[width=\textwidth]{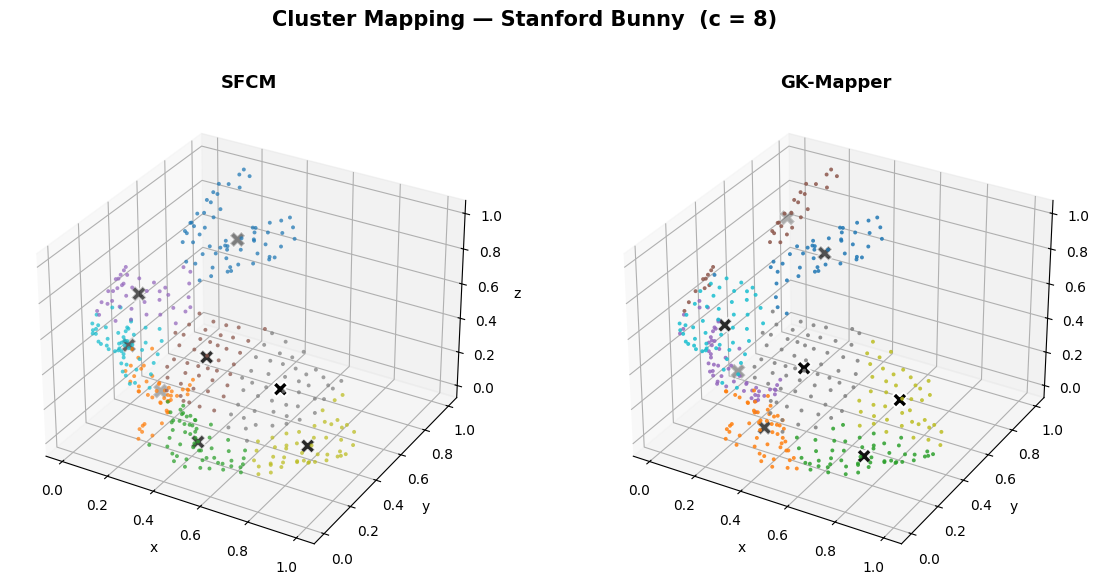}
        \caption{Cluster assignments on the 3D point cloud.}
        \label{fig:bunny_cluster_mapping}
    \end{subfigure}
    \caption{Stanford Bunny dataset.}
    \label{fig:bunny_combined}
\end{figure*}

\subsection{UCI Handwritten Digits Dataset}
\label{subsec:digits}

1797 grayscale $8\times 8$ digit images~\cite{Dua2019,Pedregosa2011},
reduced to 20 dimensions via PCA ($\approx\!95\%$ variance) and
standardized. Both methods use $c=10$, $t=0.12$, with $m_0=1.40$
(SFCM) and $m_0=2.400$ (GK-Mapper). GK-Mapper markedly raises
$t_{\mathrm{crit}}$ from $0.1000$ to $0.4386$. Since $t=0.12$
lies above the SFCM critical threshold, the SFCM graph becomes
edgeless, with $|E|=0$, $b_0=10$, and $b_1=0$. In contrast,
GK-Mapper retains a nontrivial graph with $|E|=23$, $b_0=2$,
and $b_1=15$.

The empirical stability radius is slightly larger for SFCM
($0.0004$ vs.\ $0.0003$), and SFCM shows no edge changes under
$h=0.10$, while GK-Mapper gives $|E_{\mathrm{chg}}|=11$. This
comparison must be interpreted carefully because the SFCM graph is
already edgeless at the selected threshold; therefore, the absence
of edge changes does not represent preservation of a meaningful
overlap structure. In terms of clustering metrics, GK-Mapper performs
better: the Silhouette score improves from $-0.054$ to $0.036$,
the ARI improves from $0.143$ to $0.257$, and the matching score
improves from $0.234$ to $0.439$. Thus, for UCI Digits, GK-Mapper
produces a nontrivial graph and better label agreement, although with
a slightly smaller empirical stability radius and more observed edge
changes.

\begin{figure*}[!t]
    \centering
    \begin{subfigure}[t]{0.48\textwidth}
        \centering
        \includegraphics[width=\textwidth]{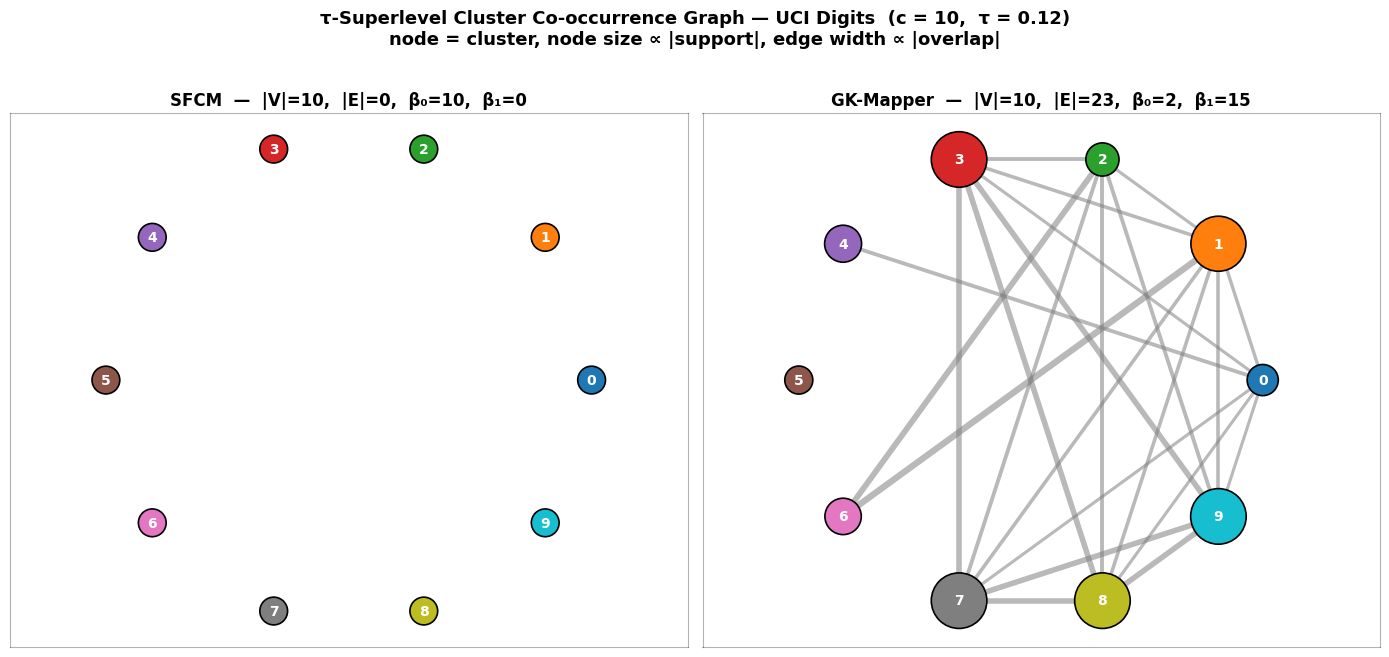}
        \caption{Co-occurrence graphs ($c=10$, $t=0.12$).
        SFCM is edgeless; GK-Mapper gives $|E|=23$, $b_1=15$.}
        \label{fig:uci_digits_mapper_graph}
    \end{subfigure}
    \hfill
    \begin{subfigure}[t]{0.48\textwidth}
        \centering
        \includegraphics[width=\textwidth]{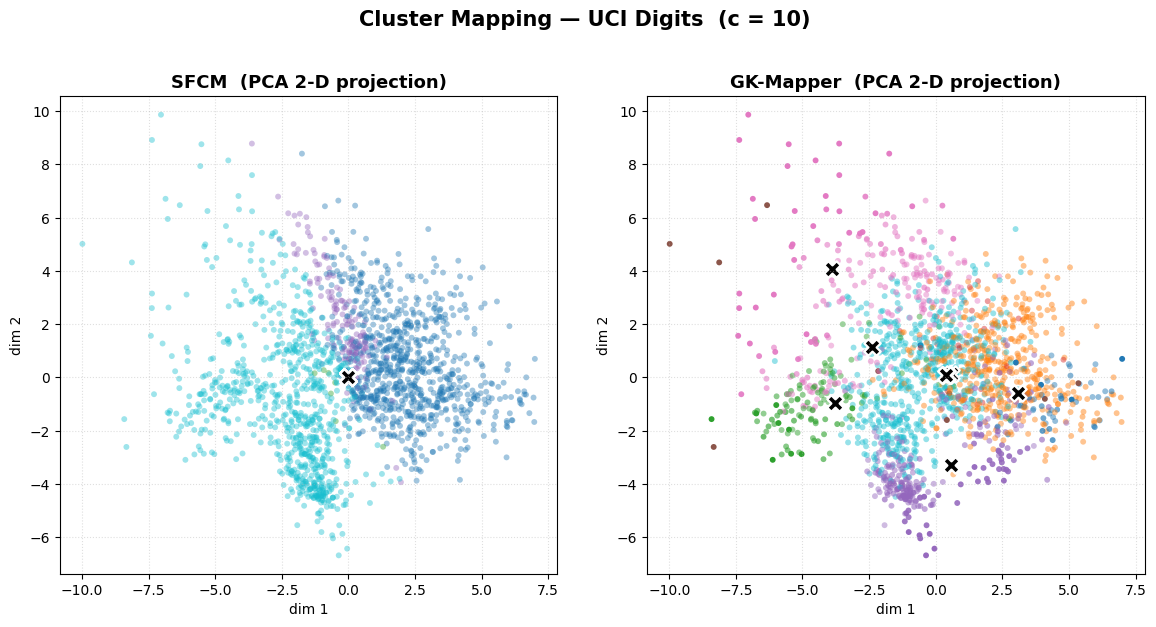}
        \caption{Cluster assignments (2D PCA projection).}
        \label{fig:uci_digits_cluster_mapping}
    \end{subfigure}
    \caption{UCI Digits dataset.}
    \label{fig:uci_digits_combined}
\end{figure*}

\subsection{Wisconsin Breast Cancer Dataset}
\label{subsec:breast_cancer}

569 samples with 30 features describing cell nuclei from digitised
fine-needle aspirates~\cite{Wolberg1995,Dua2019,Pedregosa2011};
binary malignant/benign labels. To probe the high-resolution regime
($1/c\!\approx\!0.01$), we use $c=100$, $t=0.015$, $h=0.10$,
with $m_0=1.65$ (SFCM) and $m_0=2.40$ (GK-Mapper). GK-Mapper
substantially increases $t_{\mathrm{crit}}$ from $0.0291$ to
$0.2478$ and improves the empirical stability radius from $0.0011$
to $0.1596$. It also sharply reduces the number of edge changes
from $879$ to $11$.

The graph structures are very different. SFCM produces a dense graph
with $|E|=1734$, $b_0=2$, and $b_1=1636$, whereas GK-Mapper produces
a much sparser graph with $|E|=21$, $b_0=94$, and $b_1=15$; see
Fig.~\ref{fig:breast_cancer_combined}. This indicates that the
adaptive Gustafson Kessel metric strongly reduces excessive overlap
in this high-dimensional biomedical dataset. In terms of clustering
metrics, GK-Mapper improves the Silhouette score from $-0.203$ to
$0.019$ and the matching score from $0.455$ to $0.504$, while SFCM
obtains a higher ARI ($0.256$ vs.\ $0.143$). Thus, GK-Mapper gives
a more stable and much sparser graph, although SFCM aligns better
with the binary labels under ARI.

\begin{figure*}[!t]
    \centering
    \begin{subfigure}[t]{0.48\textwidth}
        \centering
        \includegraphics[width=\textwidth]{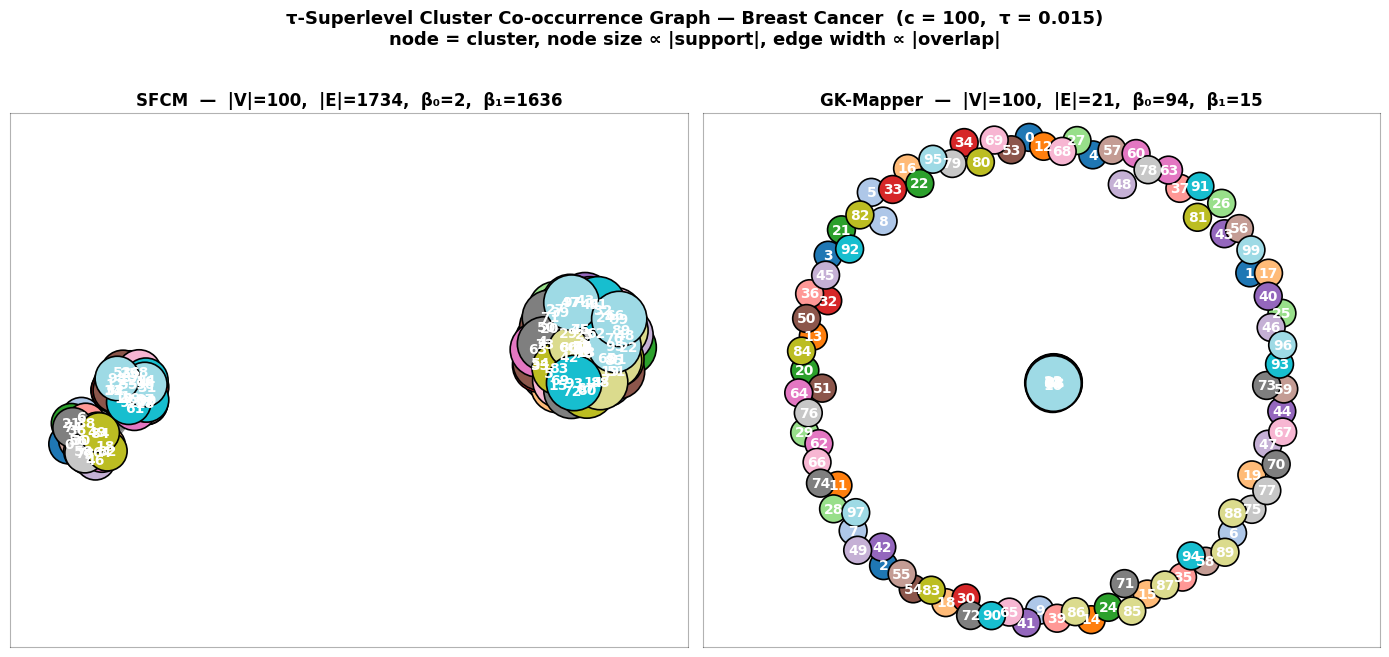}
        \caption{Co-occurrence graphs ($c=100$, $t=0.015$).
        SFCM: $|E|=1734$, $b_1=1636$;
        GK-Mapper: $|E|=21$, $b_1=15$.}
        \label{fig:breast_cancer_mapper_graph}
    \end{subfigure}
    \hfill
    \begin{subfigure}[t]{0.48\textwidth}
        \centering
        \includegraphics[width=\textwidth]{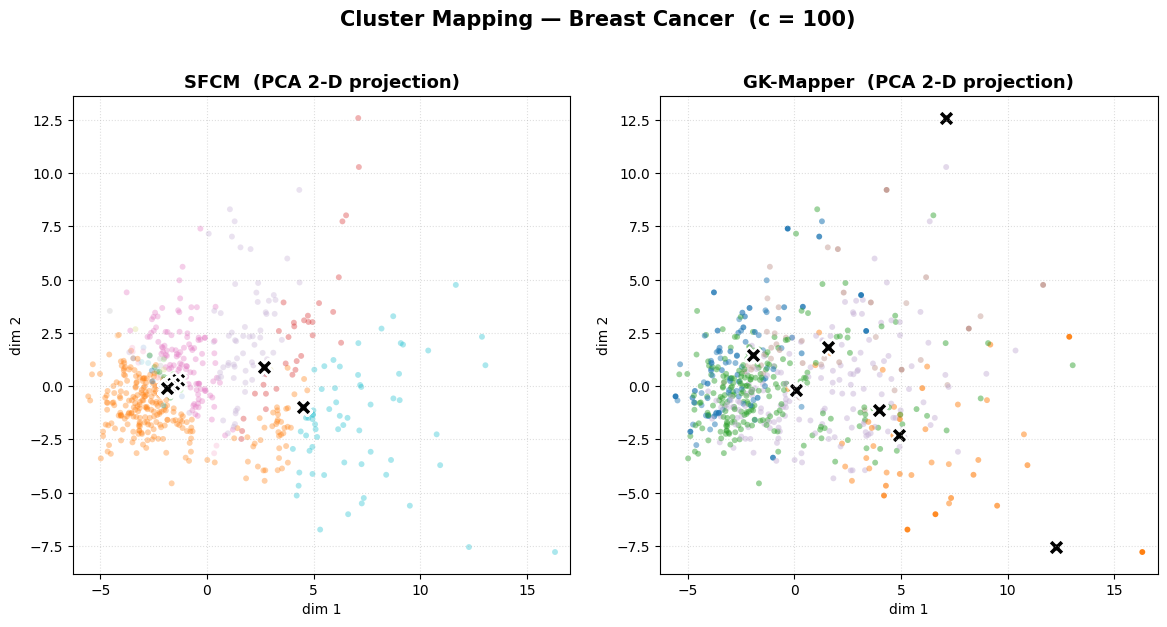}
        \caption{Cluster assignments (2D PCA projection).}
        \label{fig:breast_cancer_cluster_mapping}
    \end{subfigure}
    \caption{Wisconsin Breast Cancer dataset.}
    \label{fig:breast_cancer_combined}
\end{figure*}

\begin{table}[ht]
\centering
\caption{Summary of SFCM and GK-Mapper performance across datasets.
$r^*$: empirical stability radius; $|E_{\mathrm{chg}}|$: edge change for $h{=}0.10$.}
\label{tab:all_dataset_results}
\setlength{\tabcolsep}{3.5pt}
\renewcommand{\arraystretch}{1.05}
\footnotesize
\begin{tabular}{llccccccccccc}
\toprule
\textbf{Dataset} & \textbf{Method}
  & $t_{\mathrm{crit}}$
  & $r^{*}$
  & $m_0$
  & $|E_{\mathrm{chg}}|$
  & edges
  & $b_0$
  & $b_1$
  & Sil.
  & ARI
  & Match. \\
\midrule
\multirow{2}{*}{Breast Cancer}
  & SFCM      & 0.0291 & 0.0011 & 1.65   & 879 & 1734 & 2  & 1636 & $-$0.203 & 0.256 & 0.455 \\
  & GK-Mapper & 0.2478 & 0.1596 & 2.40   & 11  & 21   & 94 & 15   & 0.019    & 0.143 & 0.504 \\
\midrule
\multirow{2}{*}{UCI Digits}
  & SFCM      & 0.1000 & 0.0004 & 1.40   & 0   & 0    & 10 & 0    & $-$0.054 & 0.143 & 0.234 \\
  & GK-Mapper & 0.4386 & 0.0003 & 2.40   & 11  & 23   & 2  & 15   & 0.036    & 0.257 & 0.439 \\
\midrule
\multirow{2}{*}{\shortstack[l]{Anisotropic\\Ellipsoidal}}
  & SFCM      & 0.3723 & 0.0132 & 1.20   & 0   & 3    & 1  & 1    & 0.684    & 0.962 & 0.987 \\
  & GK-Mapper & 0.4390 & 0.0333 & 1.20   & 0   & 3    & 1  & 1    & 0.677    & 0.962 & 0.987 \\
\midrule
\multirow{2}{*}{Stanford Bunny}
  & SFCM      & 0.3762 & 0.0072 & 5.000  & 1   & 10   & 1  & 3    & 0.346    & NaN   & NaN   \\
  & GK-Mapper & 0.4511 & 0.0096 & 3.943  & 0   & 13   & 1  & 6    & 0.307    & NaN   & NaN   \\
\midrule
\multirow{2}{*}{Unit Circle}
  & SFCM      & 0.4016 & 0.0229 & 4.457  & 0   & 8    & 1  & 1    & 0.495    & 0.513 & 0.713 \\
  & GK-Mapper & 0.4652 & 0.0833 & 2.286  & 0   & 8    & 1  & 1    & 0.480    & 0.814 & 0.913 \\
\bottomrule

\end{tabular}
\end{table}

\section{Discussion}
\label{sec:discussion}

The empirical validation across five datasets shows that GK-Mapper
consistently increases the critical threshold $t_{\mathrm{crit}}$ compared
with SFCM. The increase is observed on the Unit Circle dataset
($0.4016$ to $0.4652$), the Anisotropic Ellipsoidal dataset
($0.3723$ to $0.4390$), the Stanford Bunny dataset
($0.3762$ to $0.4511$), the UCI Digits dataset
($0.1000$ to $0.4386$), and the Breast Cancer dataset
($0.0291$ to $0.2478$). This indicates that GK Mapper allows the graph to remain
nontrivial over a wider range of threshold values. In Theorem~\ref{thm:graveyard}, a threshold satisfying
$t>t_{\mathrm{crit}}$ makes the graph edgless and discrete.
Thus, finding $t_{\mathrm{crit}}$
is an essential first step before interpreting the resulting Mapper graph.

GK-Mapper also produces a larger empirical stability radius in four
of the five datasets. The improvement is most visible for the Unit
Circle dataset ($0.0833$ vs.\ $0.0229$), the Anisotropic Ellipsoidal
dataset ($0.0333$ vs.\ $0.0132$), the Stanford Bunny dataset
($0.0096$ vs.\ $0.0072$), and the Breast Cancer dataset
($0.1596$ vs.\ $0.0011$). The only exception is the UCI Digits
dataset, where SFCM has a slightly larger radius
($0.0004$ vs.\ $0.0003$). This supports the local stability result
of Theorem~\ref{thm:stability}: when the membership values remain
separated from the threshold $t$, small changes in the fuzzifier
$m$ do not alter the thresholded cover, and hence the graph remains
unchanged. The empirical results therefore suggest that GK-Mapper
often provides a wider local stability region, especially when the
data is heterogeneous.

The edge-change results shows that GK-Mapper reduces the number of edge changes on the Breast Cancer
dataset, where $|E_{\mathrm{chg}}|$ drops from $879$ for SFCM to $11$
for GK-Mapper. It also improves stability on the Stanford Bunny
dataset, where the edge change decreases from $1$ to $0$. On the
Anisotropic Ellipsoidal and Unit Circle datasets, both methods show
no edge changes under the selected perturbation. However, on the UCI
Digits dataset, GK-Mapper has $11$ edge changes whereas SFCM has none.
This case should be interpreted carefully, because the SFCM graph is
already edgeless at the chosen threshold, while GK-Mapper retains a
nontrivial graph with $23$ edges. Therefore, the absence of edge changes
for SFCM in this case reflects an already-empty graph rather than a
more informative stable structure. This agrees with
Theorem~\ref{thm:edge-bound}, which states that edge changes are
controlled by threshold crossings of membership values.

The clustering metrics show that graph stability and label agreement
are related but distinct objectives. GK-Mapper improves the ARI and
matching score on the Unit Circle and UCI Digits datasets. For example,
on UCI Digits, the ARI increases from $0.143$ to $0.257$, and the
matching score increases from $0.234$ to $0.439$. On the Unit Circle
dataset, the ARI increases from $0.513$ to $0.814$, and the matching
score increases from $0.713$ to $0.913$. On the Anisotropic Ellipsoidal
dataset, both methods obtain the same ARI and matching score, while
SFCM has a slightly higher silhouette value. On the Breast Cancer
dataset, GK-Mapper improves the silhouette and matching score, but
SFCM obtains a higher ARI. Therefore, a more stable or more structured
Mapper graph does not automatically imply stronger agreement with
external class labels.

Several limitations remain. The theoretical results assume local
regularity of the optimisation path but do not prove global uniqueness
or global smoothness of FCM or GK-FCM solutions. The stability guarantee
is local in the fuzzifier parameter, and the theory does not establish
a universal ordering between the stability radius of GK-Mapper and that
of SFCM. The empirical behaviour also depends on the chosen threshold,
number of clusters, initialisation, fuzzifier grid, perturbation size,
and scatter-matrix regularisation. Therefore, GK-Mapper should not be
viewed as uniformly superior to SFCM; rather, it provides a geometry-adaptive
alternative that can be more stable and more informative when the data
contain anisotropic or heterogeneous structures.

In summary, the proposed framework separates three issues in fuzzy
Mapper construction: whether the graph is nontrivial, whether it remains
locally stable under perturbations of the fuzzifier, and whether the
resulting graph agrees with external labels. The experiments indicate
that GK-Mapper usually increases the nontrivial threshold range and often
improves local stability, particularly on geometrically complex datasets.
At the same time, the comparison with SFCM remains data-dependent, so the
choice between the two methods should be guided by the geometry of the
dataset and the goal of the analysis.
\section{Conclusion}
\label{sec:conclusion}

We introduced the Gustafson Kessel Mapper (GK-Mapper)
algorithm, that generalises SFCM~\cite{bui2021sfcm} by
replacing its Euclidean cover with an Ellipsoidal cover from the Gustafson-Kessel FCM
algorithm~\cite{gustafson1979,bezdek1981}.

We then developed a
stability framework that (i) identifies the edgeless zone 
boundary $t_{\mathrm{crit}}(m)=\max_i\max_{j\neq k}
\min\{u_{ij}(m),u_{ik}(m)\}$, (ii) establishes local graph
stability near fuzzifier values, and
(iii) bounds edge changes by membership threshold crossings.
We also showed that under a single crossing condition, the critical event set
satisfies $|\mathcal{T}|\le nc$, and when $\mathcal{T}$ is
finite the graph freezes beyond $m^{**}=\sup\mathcal{T}$.

To validate these theorems, we did experiments on the Circle, Anisotropic Ellipsoidal, Stanford
Bunny, UCI Digits, and Wisconsin Breast Cancer datasets. In all these cases, GK-Mapper has yielded more stable regions, while clustering quality remains dataset-dependent. This confirms that GK-Mapper is more effective in terms of graph stability.

Future work includes (i) characterising single-crossing behaviour
along the full GK-FCM path, (ii) extending the construction to
fuzzy $c$-varieties~\cite[Section~23]{bezdek1981},
(iii) integrating persistent homology to track the filtration
induced by varying
$m$~\cite{cohen2007,edelsbrunner2000,ghrist2008}, and
(iv) deriving Lipschitz-type bounds linking membership-space
perturbations to topological
distances~\cite{lesnick2015,chazal2016}.

\section*{Author contributions}\label{sec7}
Annesha Sen and Shivam Singh contributed to the conceptualization of the problem and wrote the main manuscript text. S.P. Tiwari supervised the work and helped in the preparation of manuscript.

\section*{Declarations}\label{sec7}

{\bf Competing interests}: The authors confirm that they have no competing interests.


\end{document}